\documentclass[a4paper,10pt]{amsart}
\usepackage[margin=1.4in]{geometry}
\usepackage[T1]{fontenc}
\usepackage[latin9]{inputenc}

\usepackage{amstext}
\usepackage{amsthm}
\usepackage{amssymb}
\usepackage[foot]{amsaddr}

\usepackage{algorithm}
\usepackage{algpseudocode}
\algrenewcommand\algorithmicrequire{\textbf{Input:}}
\algrenewcommand\algorithmicensure{\textbf{Output:}}

\usepackage[shortlabels]{enumitem}

\usepackage{stmaryrd}
\usepackage[scr=boondoxo]{mathalpha} 
\usepackage{nicematrix}
\usepackage{tikz}
\usetikzlibrary{matrix}

\usepackage[sort&compress,numbers]{natbib}

\usepackage[colorlinks=true,linkcolor=black,citecolor=blue,urlcolor=blue]{hyperref}
\usepackage[nameinlink, capitalize]{cleveref}
\Crefname{corollary}{Corollary}{Corollaries}
\crefname{equation}{}{}

\AddToHook{env/lemma/begin}{\crefalias{theorem}{lemma}}
\AddToHook{env/proposition/begin}{\crefalias{theorem}{proposition}}
\AddToHook{env/corollary/begin}{\crefalias{theorem}{corollary}}

\numberwithin{equation}{section}
\numberwithin{figure}{section}

\theoremstyle{plain}
\newtheorem{theorem}{Theorem}[section]
\newtheorem{lemma}[theorem]{Lemma}
\newtheorem{corollary}[theorem]{Corollary}
\newtheorem{proposition}[theorem]{Proposition}

\crefname{figure}{figure}{figures}
\Crefname{figure}{Figure}{Figures}

\theoremstyle{remark}
\newtheorem{remark}[theorem]{Remark}

\DeclareMathOperator{\rank}{rank}
\renewcommand{\vec}[1]{\mathbf{#1}}
\newcommand{\abs}[1]{\left\vert #1 \right\vert}

\newcommand{\riesz}[1]{\mathscr{L}_{#1}}
\newcommand{\mommat}{\mathbb{M}}
\newcommand{\proj}[1]{P_{#1}}

\renewcommand{\emptyset}{\varnothing}
\newcommand{\SparseExponents}[1]{\mathcal{E}_{#1}}
\newcommand{\smon}[1]{\left\llbracket #1 \right\rrbracket^{\rm sp}}
\newcommand{\dmon}[1]{\left\llbracket #1 \right\rrbracket}
\DeclareMathOperator{\supp}{supp}
\newcommand{\pushfwd}{\sharp}

\usepackage{etoolbox}

\let\theparentequation\theequation
\patchcmd{\theparentequation}{equation}{parentequation}{}{}

\renewenvironment{subequations}[1][]{%              optional argument: label-name for (first) parent equation
  \refstepcounter{equation}%
  \setcounter{parentequation}{\value{equation}}%    parentequation = equation
  \setcounter{equation}{0}%                         (sub)equation  = 0
  \def\theequation{\theparentequation\alph{equation}}% 
  \let\parentlabel\label%                           Evade sanitation performed by amsmath
  \ifx\\#1\\\relax\else\label{#1}\fi%               #1 given: \label{#1}, otherwise: nothing
  \ignorespaces
}{%
  \setcounter{equation}{\value{parentequation}}%    equation = subequation
  \ignorespacesafterend
}

\newcommand*{\nextParentEquation}[1][]{%            optional argument: label-name for (first) parent equation
  \refstepcounter{parentequation}%                  parentequation++
  \setcounter{equation}{0}%                         equation = 0
  \ifx\\#1\\\relax\else\parentlabel{#1}\fi%         #1 given: \label{#1}, otherwise: nothing
}

\crefname{parentequation}{}{}
\creflabelformat{parentequation}{#2\textup{(#1)}#3}% for nameinling=true
\usepackage{trimclip}
\makeatletter
\DeclareRobustCommand{\shortto}{%
  \mathrel{\mathpalette\short@to\relax}%
}
\newcommand{\short@to}[2]{%
  \mkern2mu
  \clipbox{{.4\width} 0 0 0}{$\m@th#1\vphantom{+}{\rightarrow}$}%
  }
\makeatother

\usepackage[font=small]{caption}

\title[Finite convergence in sparse moment relaxations]{Finite convergence and minimizer extraction in moment relaxations with correlative sparsity}
\author{Giovanni Fantuzzi$^1$}
\address{\normalfont$^1$Department of Mathematics, Friedrich--Alexander Universit\"at Erlangen--N\"urnberg}
\email{\href{mailto:giovanni.fantuzzi@fau.de}{giovanni.fantuzzi@fau.de}}

\author{Federico Fuentes$^2$}
\address{\normalfont$^2$Institute for Mathematical and Computational Engineering (IMC), School of Engineering and Faculty of Mathematics, Pontificia Universidad Cat\'olica de Chile}
\email{\href{mailto:federico.fuentes@uc.cl}{federico.fuentes@uc.cl}}

\date{}

\begin{document}
\begin{abstract}
We identify a new sufficient condition for the finite convergence of moment relaxations of polynomial optimization problems with correlative sparsity. This condition, which follows from a solution to a correlatively sparse version of the classical truncated moment problem, requires that certain moment matrices admit a flat extension and that the variable cliques underpinning the relaxation satisfy a `running intersection' property. We also describe an algorithm that, when these conditions are met, extracts at least as many minimizers for the original polynomial optimization problem as the largest rank of the moment matrices in its relaxation. Our results, along with the necessity of the running intersection property, are illustrated with examples.
\end{abstract}

\maketitle

%% Main matter
\vspace{-10pt}
\section{Introduction}

Moment-sum-of-squares (moment-SOS) relaxations are an established approach to bound from below the global minimum of a polynomial optimization problem (POP) by solving semidefinite programs \cite{Lasserre2001,Parrilo2003,Laurent2009}. 
Raising the relaxation order to improve the lower bound rapidly increases the size of these semidefinite programs, resulting in high or outright prohibitive computational costs and memory footprint. 
To avoid this, one can modify standard relaxations to exploit properties of the POP such as symmetries \cite{Gatermann2004,Lofberg2009,Riener2013,BlekhermanRiener2020} and sparsity \cite{Waki2006,Lasserre2006,Wang2019,Wang2020a,Wang2020b,Wang2020c,Zheng2021,MagronWang2023}.
A key question is whether these modified relaxations enjoy the same good theoretical properties as the standard ones. 
In particular, one would like to have asymptotic convergence results for the lower bounds, sufficient conditions to detect finite convergence, and algorithms to extract one or more POP minimizers when these conditions hold. 

It is well known that symmetry-reduced moment-SOS hierarchies are equivalent to the original one, and the same is true of hierarchies that exploit \emph{term} sparsity if they are constructed using maximal chordal extensions of the underlying sparsity patterns \cite{Wang2020a,MagronWang2023}. Here, we consider POPs with \emph{correlative} sparsity, identifying new sufficient conditions to detect the finite convergence of moment relaxations of ~\cite{Waki2006,Lasserre2006}. We briefly review these relaxations next, and give a precise account of our contributions in \cref{ss:intro-results}. The relation between our results and existing ones is explained in \cref{ss:intro-discussion}.

\subsection{Correlatively sparse POPs and their moment relaxations}
Denote the non-negative integers by $\mathbb{N}$. Fix $n\in\mathbb{N}$ positive and cover the set $[n]:=\{1,\ldots,n\}$ by subsets $\Delta_1,\ldots,\Delta_m \subseteq [n]$ called \emph{cliques}. We refer to the set of cliques as the \emph{clique cover}. We assume without loss of generality that no clique is contained in another clique. Given a `global' vector $\vec{x} \in \mathbb{R}^n$, let $\vec{x}_{\Delta_i}$ be the `clique subvector' indexed by the elements of $\Delta_i$. 
A POP is correlatively sparse if it has the form
\begin{equation}\label{e:pop}\tag{POP}
        f^* := \min_{\vec{x} \in \mathbb{R}^n}\quad 
        \sum_{i=1}^m f_i(\vec{x}_{\Delta_i})
        \quad\text{s.t.}\quad
        \vec{g}_i(\vec{x}_{\Delta_i})\geq 0 \quad \forall i \in [m]
\end{equation}
for polynomials $f_i:\mathbb{R}^{\abs{\Delta_i}}\to \mathbb{R}$ and $\vec{g}_i:\mathbb{R}^{\abs{\Delta_i}}\to \mathbb{R}^{N_i}$ with $N_i \in \mathbb{N}$. The inequality constraints are understood element-wise. 
To simplify notation in what follows, we set $N_i=1$ and $\vec{g}_i(\vec{x}_{\Delta_i}) = 1$ if $\vec{x}_{\Delta_i}$ is unconstrained even though these constraints are redundant.

Fix a \emph{relaxation order} $\omega \in \mathbb{N}$ satisfying $2\omega \geq \max \{1,\deg f, \deg \vec{g}_1,\ldots,\deg \vec{g}_m\}$, where \mbox{$f(\vec{x})\!:=\!\sum_{i=1}^m \!f_i(\vec{x}_{\Delta_i})$}. Let $\SparseExponents{2\omega}$ be the set of multi-indices $\alpha \in \mathbb{N}^n$ of degree $\alpha_1 + \cdots + \alpha_n \leq 2\omega$ and support contained in at least one clique (see \cref{ss:csp-stuff} for a precise definition). We call $\SparseExponents{2\omega}$ the set of \emph{correlatively sparse multi-indices} subordinate to the clique cover.  Let $\vec{f}$ be the vector of coefficients of $f(\vec{x})$ with respect to the sparse monomial vector 
$\smash{\smon{\vec{x}}_{2\omega} = \left( \vec{x}^\alpha \right)_{\alpha \in \SparseExponents{2\omega}}}$. 
Finally, let $\smash{\mommat^{\omega}_{\Delta_{i}}(\vec{y})}$ and $\smash{\mommat^{\omega}_{\Delta_{i}}(\vec{g}_i\vec{y})}$ denote the usual \emph{moment} and \emph{localizing matrices} of order $\omega$ associated with the clique $\Delta_i$ (see \cref{ss:mommat} for definitions).
Using angled brackets to indicate inner products, the sparse moment relaxation of \cref{e:pop} of order $\omega$ is
\begin{equation}\label{e:mom}\tag{MOM}
    f_\omega^* := \min_{ \vec{y} \in \mathbb{R}^{|\SparseExponents{2\omega}|} }\quad \langle \vec{f}, \vec{y} \rangle 
    \quad\text{s.t.}\quad 
    \begin{cases}
    \mommat^\omega_{\Delta_{i}}(\vec{y}) \succeq 0 &\forall i\in[m],\\
    \mommat^\omega_{\Delta_{i}}(\vec{g}_i\vec{y}) \succeq 0 &\forall i\in[m],\\
    y_{\vec{0}} = 1.
    \end{cases}
\end{equation}
It is known that $\smash{f_\omega^* \leq f^*}$ \cite{Waki2006,Lasserre2006}. Lasserre \cite{Lasserre2006} showed that $\smash{f_\omega^* \to f^*}$ as $\omega\to\infty$ if the constraint polynomials $\vec{g}_i$ generate Archimedean quadratic modules (see, for example, \mbox{\cite[Definition~3.18]{Laurent2009}}) and if the cliques $\Delta_1,\ldots,\Delta_m$ satisfy the \emph{running intersection property}:
\begin{equation}\label{e:rip}\tag{RIP}
    \forall i > 1, \quad
    \exists j \in [i-1]
    \quad\text{such that}\quad
    \Delta_{i}\cap(\Delta_{1}\cup\cdots\cup\Delta_{i-1})\subseteq \Delta_{j}.
\end{equation}
Estimates for the rate of this asymptotic convergence also exist~\cite{Korda2025}. As we explain next, we identify sufficient rank conditions for the moment matrices, and certain submatrices thereof, that imply $\smash{f_\omega^* = f^*}$ for finite $\omega$.

\subsection{Main results}\label{ss:intro-results}
Let the feasible set of \cref{e:pop} be
\begin{equation}\label{e:bsa}
    K := \left\{ \vec{x}\in\mathbb{R}^n:\quad \vec{g}_i(\vec{x}_{\Delta_i}) \geq 0 \quad \forall i\in[m]\right\}.
\end{equation}
Our first contribution is to provide practical sufficient conditions to answer a correlatively sparse version of the truncated moment problem on $K$, which asks whether the entries of a vector $\smash{\vec{y}=(y_\alpha)_{\alpha\in \SparseExponents{2\omega}}}$ are the moments of a positive Borel measure (called a \emph{representing measure}) supported on $K$. 
Specifically, let $\smash{\mommat_{\Delta_{i}\cap \Delta_{j}}^{\omega}(\vec{y})}$ be the largest principal submatrix of the moment matrix $\smash{\mommat_{\Delta_{i}}^{\omega}(\vec{y})}$ that, up to a symmetric permutation, is also a principal submatrix of $\smash{\mommat_{\Delta_{j}}^{\omega}(\vec{y})}$ (see \cref{ss:mommat} for a precise definition). Set
\begin{equation}\label{e:di-def}
    d_i = \max \left\{ 1, \left\lceil \tfrac12 \deg\vec{g}_i \right\rceil \right\},
    \quad i \in [m]
\end{equation}
and let $\delta_{\vec{x}}$ be the Dirac measure at $\vec{x}$. 
In \cref{s:proof} we prove the following statement.

\begin{theorem}\label{thm:FF}
    Fix a nonzero vector $\vec{y} =(y_\alpha)_{\alpha\in \SparseExponents{2\omega}}$.
    Let the cliques $\Delta_1,\ldots,\Delta_m$ satisfy~\cref{e:rip}. Assume that
    \begin{alignat}{3}
        \label{e:lmis}
        \mommat^\omega_{\Delta_{i}}(\vec{y})\succeq 0
        \quad &\text{and} \quad 
        \mommat^\omega_{\Delta_{i}}(\vec{g}_i\vec{y}) \succeq 0  \quad &&\forall i \in [m],
        \\
        \label{e:flat-moments}
        \rank\mommat_{\Delta_{i}}^{\omega}(\vec{y}) &= \rank\mommat_{\Delta_{i}}^{\omega-d_{i}}(\vec{y}) \quad &&\forall i \in [m].
    \end{alignat}
    Assume also that for every $i \in\{2,\ldots,m\}$ there exists $j \in [i-1]$ satisfying \cref{e:rip} such that
    \begin{equation}\label{e:flat-overlap}
        \rank\mommat_{\Delta_{i}\cap \Delta_{j}}^{\omega}(\vec{y}) = \rank\mommat_{\Delta_{i}\cap \Delta_{j}}^{\omega-1}(\vec{y}).
    \end{equation}
    Then, there exist an integer $r\geq \max_{i\in[m]} \rank\mommat_{\Delta_{i}}^{\omega}(\vec{y})$, points $\vec{x}_1,\ldots,\vec{x}_r$ of $K$, and scalars $\lambda_1,\ldots,\lambda_r > 0$ such that the atomic measure $\lambda_1 \delta_{\vec{x}_1} + \cdots + \lambda_r \delta_{\vec{x}_r}$ represents~$\vec{y}$. In particular,
    \begin{equation}\label{e:atomic-y}
        \vec{y} = 
        \lambda_1 \smon{ \vec{x}_1 }_{2\omega} 
        + \cdots 
        + \lambda_r \smon{ \vec{x}_r }_{2\omega}.
    \end{equation}
\end{theorem}

We call conditions \cref{e:flat-moments,e:flat-overlap} `flat rank' conditions (sometimes also referred to as `flat extension' or `flat truncation' conditions), since they imply that the moment matrix $\smash{\mommat_{\Delta_{i}}^{\omega}(\vec{y})}$ can be truncated to the submatrix $\smash{\mommat_{\Delta_{i}}^{\omega-d_{i}}(\vec{y})}$ without changing its rank.
Using them, we can derive the sufficient condition for the finite convergence of \cref{e:mom} announced above. The proof is short, so we give it immediately.

\begin{corollary}\label[corollary]{thm:FF-POP}
Let the cliques $\Delta_1,\ldots,\Delta_m$ satisfy \cref{e:rip}.
If $\vec{y}$ is optimal for~\cref{e:mom} and satisfies the flat rank conditions \cref{e:flat-moments,e:flat-overlap}, then $f^*_\omega=f^*$ and the points $\vec{x}_1,\ldots,\vec{x}_r$ in the atomic decomposition \cref{e:atomic-y} solve~\cref{e:pop}.
\end{corollary}

\begin{proof}
    The conditions of \Cref{thm:FF} are clearly satisfied by assumption, so $\vec{y}$ admits the atomic decomposition~\cref{e:atomic-y} with points $\vec{x}_1,\ldots,\vec{x}_r \in K$ and scalars $\lambda_1,\ldots,\lambda_r > 0$. In particular,
    $\lambda_1 + \cdots +\lambda_r = y_{\vec{0}} = 1$. Then, since $f(\vec{x}) = \langle \vec{f}, \smon{\vec{x}}_{2\omega}\rangle$ and since the points $\vec{x}_1,\ldots,\vec{x}_r$ are feasible for \cref{e:pop}, we obtain
    \begin{equation}
        f^*
        = \sum_{i=1}^r \lambda_i f^*
        \leq \sum_{i=1}^r \lambda_i f(\vec{x}_i)
        = \sum_{i=1}^r \lambda_i \langle \vec{f}, \smon{\vec{x}_i}_{2\omega}\rangle
        \overset{{\text{\cref{e:atomic-y}}}}{=} \langle \vec{f}, \vec{y}\rangle
        = f_\omega^* \leq f^*.
    \end{equation}
    All inequalities must clearly be equalities, giving $f_\omega^* = f^*$. Since $\lambda_1,\ldots,\lambda_r>0$, we must also have $f(\vec{x}_1) = \cdots = f(\vec{x}_r) = f^*$, so $\vec{x}_1,\ldots,\vec{x}_r$ are optimal for \cref{e:pop}.
\end{proof}

\begin{remark}
    As in \cite[Theorem~3.5]{Nie2024}, in \cref{thm:FF-POP} it suffices to check the flat rank conditions \cref{e:flat-moments,e:flat-overlap} using the matrices $\smash{\mommat^t_{\Delta_i}(\vec{y})}$ and $\smash{\mommat^t_{\Delta_i\cap\Delta_j}(\vec{y})}$ for some $\smash{t\in[\omega_0, \omega]}$, where $\smash{\omega_0}$ is the smallest integer satisfying $\smash{2\omega_0 \geq \max \{1,\deg f, \deg \vec{g}_1,\ldots,\deg \vec{g}_m\}}$. The proof goes through because the vector $\smash{(y_\alpha)_{\alpha \in \SparseExponents{2t}}}$ has an atomic representing measure by \cref{thm:FF}, and its atoms $\vec{x}_i$ satisfy $f(\vec{x}_i) = \langle \vec{f}, \smon{\vec{x}_i}_{2t}\rangle=\langle \vec{f}, \smon{\vec{x}_i}_{2\omega}\rangle=\langle \vec{f}, \vec{y}\rangle$.
\end{remark}

Before discussing how \Cref{thm:FF} and \Cref{thm:FF-POP} relate to existing results in the literature, some technical remarks are in order. 
Firstly, as demonstrated in \cref{ss:counterexample} and by \cite[Example~3.4]{Nie2024}, one cannot drop \cref{e:rip} without additional assumptions. 
Secondly, in some cases when \cref{e:rip} does not hold, the cliques can be reordered so that it does. Our notation $\Delta_1,\ldots,\Delta_m$ always refers to the reordered cliques and we say that \cref{e:rip} fails if no clique reordering can enforce it. 
Thirdly, while the asymptotic convergence result of~\cite{Lasserre2006} requires $K$ to be compact, \Cref{thm:FF} and \Cref{thm:FF-POP} apply also in the noncompact case.
Finally, in contrast to the similar results for the classical $K$-moment problem (see \cite{Fialkow2016review,Laurent2009} for reviews of the subject), the representing measure in \Cref{thm:FF} need not be unique. For example, with $n=3$ and cliques $\Delta_1=\{1,2\}$ and $\Delta_2 = \{2,3\}$,  the measures
\begin{equation}\label{e:mu-lambda}
    \mu_\lambda = 
      \lambda\, \delta_{(1,0,1)}
    + \left( \tfrac{1}{2} - \lambda \right) \delta_{(1,0,-1)}
    + \left( \tfrac{1}{2} - \lambda \right) \delta_{(-1,0,1)}
    + \lambda \delta_{(-1,0,-1)}
\end{equation}
generate the same correlatively sparse moment vector $(y_\alpha)_{\alpha\in \SparseExponents{2\omega}}$  for every $\smash{\lambda \in [0,\frac12]}$ irrespective of $\omega\in\mathbb{N}$.
This lack of uniqueness arises because a correlatively sparse moment sequence only determines the $\vec{x}_{\Delta_i}$-marginals of a representing measure, which in general do not uniquely characterize the representing measure itself.

\subsection{Relation to previous work} \label{ss:intro-discussion}
The conditions for the finite convergence of \cref{e:mom} in \Cref{thm:FF-POP} generalize those in \cite[Theorem~3.5]{Nie2024}, which restricts the matrix ranks in \cref{e:flat-moments,e:flat-overlap} to coincide and asks for \cref{e:flat-overlap} to hold for all pairs of intersecting cliques. These stronger conditions hold if the atoms in the atomic decomposition \cref{e:atomic-y} have pairwise distinct projections onto all clique intersections, but problem structures such as symmetries may force different ranks. \Cref{ss:ex-pop} gives an example in which \cref{thm:FF-POP} detects finite convergence, while \cite[Theorem~3.5]{Nie2024} does not. Our result differs also from \cite[Theorem~3.7]{Lasserre2006}, which does not assume \cref{e:rip} but strengthens \cref{e:flat-overlap} by imposing that $\rank\mommat_{\Delta_{i}\cap \Delta_{j}}^{\omega}(\vec{y})=1$ for every pair of intersecting cliques. This is a strong restriction because it requires the vectors $\vec{x}_1,\ldots,\vec{x}_r$ to have the same entries in positions indexed by clique intersections, which in our experience rarely happens in practice. On the other hand, we are not aware of other sufficient flat rank conditions that certify finite convergence for \cref{e:mom} but do not require~\cref{e:rip}. 

Flat rank conditions are not the only way to detect the finite convergence of \cref{e:mom}: any result ensuring the existence of a minimizer $\smash{\vec{x}^*}$ for \cref{e:pop} with $\smash{f(\vec{x}^*)=f_\omega^*}$ suffices. One example is \cite[Theorem~3.3]{Nie2024}, which certifies that $\smash{\vec{x}^* \in\mathbb{R}^n}$ solves \cref{e:pop} if its clique projections $\smash{\vec{x}^*_{\Delta_1}},\ldots,\smash{\vec{x}^*_{\Delta_m}}$ are in the support of atomic measures representing the `clique subvectors' of $\vec{y}$ indexed by the multi-indices $\alpha \in \SparseExponents{2\omega}$ supported on individual cliques. 
As we will see below, when the \cref{e:rip} and flat rank conditions hold, the assumptions of \mbox{\cite[Theorem~3.3]{Nie2024}} are satisfied, so \cite[Theorem~3.3]{Nie2024} and \cref{thm:FF-POP} become equivalent.
On the other hand, \cite[Theorem~3.3]{Nie2024} does not require \cref{e:rip}, so in principle it can certify finite convergence when \cref{thm:FF-POP} cannot. In practice, however, verifying its assumptions can be difficult because there is no systematic way to construct $\smash{\vec{x}^*}$ from an optimal solution of \cref{e:mom}. 
Indeed, without \cref{e:rip}, a suitable $\smash{\vec{x}^*}$ may not exist even when the rank conditions in \cref{e:flat-moments,e:flat-overlap} hold for all cliques and all pairs of intersecting cliques (see, for instance, \cref{ss:counterexample} and \cite[Example~3.4]{Nie2024}).

Moving to the more general truncated correlatively sparse $K$-moment problem, we are not aware of explicit results similar to our \cref{thm:FF}. In fact, as we now explain, a secondary contribution of our work is to recognize that prior convergence analyses of \cref{e:mom} implicitly rely on results similar to our \Cref{thm:FF}, which however are not stated explicitly. For example, \cite[Theorem 3.5]{Nie2024} follows from our \Cref{thm:FF} if one requires the matrix ranks in \cref{e:flat-moments,e:flat-overlap} to coincide. Likewise, Lasserre's proof of the asymptotic convergence of \cref{e:mom}~\cite{Lasserre2006} tacitly uses the following solution to the \emph{full} correlatively sparse $K$-moment problem, where the finite vector $\vec{y} = (y_\alpha)_{\alpha \in \SparseExponents{2\omega}}$ is replaced by the infinite sequence $\vec{y} = (y_\alpha)_{\alpha \in \SparseExponents{\infty}}$. (We attribute the statement to Lasserre even though it does not appear in \cite{Lasserre2006}, because it is proven in \cite[\S4.1]{Lasserre2006} as part of the proof of \cite[Theorem~3.6]{Lasserre2006}.)

\begin{theorem}[Lasserre]\label{thm:lasserre}
    Fix a nonzero sequence $\vec{y} = (y_\alpha)_{\alpha \in \SparseExponents{\infty}}$. Assume that:
    \begin{enumerate}[{\rm(a)}]
    \item\label{ass:lasserre:rip} The cliques $\Delta_1,\ldots,\Delta_m$ satisfy \cref{e:rip};
    \item\label{ass:lasserre:lmis} For every $\omega \in \mathbb{N}$ and every $i\in[m]$, $\vec{y}$ satisfies $\smash{\mommat^{\omega}_{\Delta_{i}}(\vec{y})} \succeq 0$ and $\smash{\mommat^{\omega}_{\Delta_{i}}(\vec{g}_i\vec{y})} \succeq 0$;
    \item\label{ass:lasserre:archimedean} For every $i\in[m]$, there exist an  integer $R_i$ and sums-of-squares polynomials $\sigma_0,\ldots,\sigma_{N_i}:\mathbb{R}^{\abs{\Delta_i}}\to \mathbb{R}$ such that $\sigma_0(\vec{z}) + \sum_{j=1}^{N_i} g_{i,j}(\vec{z})\sigma_{j}(\vec{z}) = R_i^2 - |\vec{z}|^2$.
    \end{enumerate}
    Then, $\vec{y}$ has a representing measure supported on $K$.
\end{theorem}

Separating the solution of sparse moment problems from their application to moment-SOS relaxations also helps us clarify their proofs. For example, to prove \Cref{thm:lasserre}, one first applies to each clique $\Delta_i$ a measure-theoretic version of Putinar's Positivstellensatz \mbox{\cite[Lemma~4.1]{Putinar1993}}, which states that the `clique subsequence' of $\vec{y}$ indexed by the multi-indices $\alpha \in \SparseExponents{\infty}$ supported on $\Delta_i$ has a unique representing measure $\mu_{\Delta_i}$ when assumptions \ref{ass:lasserre:lmis} and \ref{ass:lasserre:archimedean} hold. Then, one verifies that the measures $\mu_{\Delta_1},\ldots,\mu_{\Delta_m}$ have consistent marginals on the clique intersections. Finally, one assembles these consistent `local' measures into the `global' representing measure $\mu$ through a technical construction based on \cref{e:rip} via~\mbox{\cite[Lemma~6.4]{Lasserre2006}}.

\Cref{thm:FF} can  be proven with
the same strategy because, as we will show, its assumptions guarantee that the `clique subvectors' of $\vec{y} = (y_\alpha)_{\alpha \in \SparseExponents{2\omega}}$ have representing measures $\mu_{\Delta_1},\ldots,\mu_{\Delta_m}$ with consistent marginals.
Specifically, conditions \cref{e:lmis} and \cref{e:flat-moments} enable a \emph{flat extension} of the moment matrices for each clique and ensure the existence of \emph{finitely atomic} representing measures $\mu_{\Delta_1},\ldots,\mu_{\Delta_m}$ for the clique subvectors (see \cite{CurtoFialkow2000,Fialkow2016review}, as well as \cite{Laurent2009} for an introductory exposition). Condition \cref{e:flat-overlap} ensures that these measures have consistent marginals.
We provide the technical details in \cref{s:proof}. 

In \cref{s:recovery}, we give an elementary and explicit version of Lasserre's measure assembly procedure for our atomic setting. This contribution is important for two reasons. First, it results in a new, simple algorithm for recovering the `atomic decomposition' \cref{e:atomic-y} in practice (see \cref{alg:measure_recovery}). This broadens the class of POPs for which minimizers can be extracted algorithmically. Second, our explicit construction enables us to prove that this algorithm returns an atomic representing measure with \emph{maximal support} (in the language of \cite{Fialkow2017corevariety,Blekherman2020corevariety}, it constructs the \emph{core variety} of the Riesz functional associated with $\vec{y} = (y_\alpha)_{\alpha \in \SparseExponents{2\omega}}$). In the context of moment-SOS relaxations of POPs, this means our algorithm extracts \emph{all} POP minimizers encoded by the solution of \cref{e:mom}.
As we explain, however, `sparser' representing measures can also be obtained if desired, using convex programming.

Finally, let us mention that many of the sufficient conditions to detect the finite convergence of \cref{e:mom} from \cite{Lasserre2006,Nie2024} were recently generalized to sparse polynomial matrix optimization problems \cite{Miller2024sparsePMI}. We expect analogues of \Cref{thm:FF} and \Cref{thm:FF-POP} to hold in this setting, too, but leave the problem to future work.

%%%%%%%%%%%%%%%%%%%%%%%%%%%%%%%%%%%%%%%%
\section{Preliminaries}
Before proving \Cref{thm:FF}, we define more precisely the notion of correlatively sparse multi-indices and monomials already introduced above. We also recall the definition of the moment and localizing matrices appearing in \Cref{thm:FF}, and we introduce some `clique projection' operators that will be useful in its proof.

%%%%%%%%%%%%%%%%%%%%%%%%%%%%%%%%%%%%%%%%
\subsection{Correlatively sparse multi-indices and monomials}\label{ss:csp-stuff}
We say that a multi-index $\alpha = (\alpha_1,\ldots,\alpha_n) \in \mathbb{N}^n$ is supported on a clique $\Delta_i$, written $\supp\alpha \subseteq \Delta_i$, if $\alpha_j=0$ for all indices $j \notin \Delta_i$. For every $\omega \in \mathbb{N}$, the subset of \emph{correlatively sparse} multi-indices of degree up to $2\omega$ subordinate to the clique cover $\Delta_1,\ldots,\Delta_m$ is
\begin{equation}
    \SparseExponents{2\omega} :=
    {\textstyle \bigcup\limits_{i=1}^m}
    \left\{\alpha \in \mathbb{N}^n:\;\supp \alpha \subseteq \Delta_i, \quad \alpha_1 + \cdots + \alpha_n \leq 2\omega \right\}.
\end{equation}
We write $\smash{\smon{\vec{x}}_{2\omega} = \left( \vec{x}^\alpha \right)_{\alpha \in \SparseExponents{2\omega}}}$ for the vector of correlatively sparse monomials.

%%%%%%%%%%%%%%%%%%%%%%%%%%%%%%%%%%%%%%%%
\subsection{Moment and localizing matrices}\label{ss:mommat}

Fix $\omega \in \mathbb{N}$.
The \emph{Riesz functional} associated with a vector $\vec{y} = (y_\alpha)_{\alpha \in \SparseExponents{2\omega}}$, denoted by $\riesz{\vec{y}}$, is the unique continuous linear functional on the polynomial space spanned by the correlatively sparse monomials $\smon{\vec{x}}_{2\omega}$ that satisfies
$\riesz{\vec{y}}(\vec{x}^\alpha) = y_\alpha$ for every exponent $\alpha \in \SparseExponents{2\omega}$. 

For every $i\in [m]$, recall that $\vec{x}_{\Delta_i}$ is the subvector of $\vec{x}$ indexed by clique $\Delta_i$. We list the full set of monomials with indeterminate $\vec{x}_{\Delta_i}$ and total degree up to $d$ in the vector
\begin{equation}
    \dmon{\vec{x}_{\Delta_i}}_{d}:= 
    \left( \vec{x}_{\Delta_i}^\alpha \right)_{\alpha\in \mathbb{N}^{|\Delta_i|}_d}
    \quad\text{with}\quad
    \mathbb{N}^{|\Delta_i|}_d:=
    \bigg\{\alpha \in \mathbb{N}^{|\Delta_i|}:\;\sum_{j=1}^{|\Delta_i|} \alpha_j \leq d
    \bigg\}.
\end{equation}
The moment matrix of order $d \leq \omega$ associated with clique $\Delta_i$ is then defined as
\begin{equation}
    \mommat^{d}_{\Delta_i}(\vec{y}) := 
    \riesz{\vec{y}}\Big( \dmon{ \vec{x}_{\Delta_i} }_{d} \otimes \dmon{ \vec{x}_{\Delta_i} }_{d} \Big),
\end{equation}
where $\otimes$ denotes the usual tensor product. The `overlap' moment matrix $\smash{\mommat_{\Delta_{i}\cap \Delta_{j}}^{d}(\vec{y})}$ is defined similarly whenever $\Delta_i$ intersects $\Delta_j$ and we define $\smash{\mommat_{\emptyset}^{d}(\vec{y})} = \riesz{\vec{y}}(1)$.
Letting $\smash{d_{i,j} = \left\lceil \frac12 \deg g_{i,j} \right\rceil}$, the localizing matrix of order $d \in [d_{i,j}, \omega]$ associated with the clique $\Delta_i$ and the entry $g_{i,j}$ of the polynomial vector $\vec{g}_i:\mathbb{R}^{|\Delta_i|} \to \mathbb{R}^{N_i}$ is
\begin{equation}
    \mommat^{d}_{\Delta_i}(g_{i,j}\vec{y}) := 
    \riesz{\vec{y}}\Big( 
    g_{i,j} (\vec{x}_{\Delta_i}) \,
    \big( \dmon{ \vec{x}_{\Delta_i} }_{d-d_{i,j}} \otimes \dmon{ \vec{x}_{\Delta_i} }_{d-d_{i,j}} \big)
    \Big).
\end{equation}
If $d \in [d_i,\omega]$ with $d_i$ defined as in \cref{e:di-def}, we collect these localizing matrices in the block-diagonal matrix 
\begin{equation}
    \mommat^{d}_{\Delta_i}(\vec{g}_i\vec{y})
    =\begin{bmatrix}
        \mommat^{d}_{\Delta_i}(g_{i,1}\vec{y}) \\
        &\ddots \\
        && \mommat^{d}_{\Delta_i}(g_{i,N_i}\vec{y})
    \end{bmatrix},
\end{equation}
which serves as the localizing matrix associated with the polynomial vector $\vec{g}_i$.
Note that, for fixed $i\in[m]$, the moment, overlap, and localizing matrices depend not on the full vector $\vec{y}$, but only on its subvector
\begin{equation}\label{e:local-moments}
    \vec{y}_{\Delta_i} := 
    \riesz{\vec{y}}\big(  \dmon{ \vec{x}_{\Delta_i} }_{2\omega} \big).
\end{equation}
These subvectors list the entries $y_\alpha$ indexed by multi-indices $\alpha \in \SparseExponents{2\omega}$ supported on the clique $\Delta_i$. 

\subsection{Clique projections and pushforwards}
For any index set $\Delta \subseteq [n]$, denote by $\vec{x}_\Delta$ the subvector of $\vec{x}$ indexed by the elements of $\Delta$. Let 
$\smash{\proj{\Delta}:\mathbb{R}^n\to\mathbb{R}^{|\Delta|}}$ 
be the projection operator satisfying
$\proj{\Delta} \vec{x}=\vec{x}_\Delta$ 
and denote its pseudoinverse by 
$\smash{\proj{\Delta}^\dagger:\mathbb{R}^{|\Delta|} \to \mathbb{R}^n}$.
We can identify $\proj{\Delta}$ with the $|\Delta| \times n$ matrix listing the rows of the $n\times n$ identity matrix indexed by $\Delta$, and $\smash{\proj{\Delta}^\dagger}$ with its transpose. Note that $\smash{\proj{\Delta}^\dagger\vec{z}}$ simply `lifts' a vector $\vec{z}\in\mathbb{R}^{|\Delta|}$ into a vector in $\mathbb{R}^n$ by placing $\vec{z}$ in the entries indexed by $\Delta$ and padding the rest with zeros. 

Given $\Delta' \subseteq \Delta \subseteq [n]$, let $\pi_{\Delta \shortto \Delta'}: \mathbb{R}^{|\Delta|} \to \mathbb{R}^{|\Delta'|}$ be the projection operator defined~by
\begin{equation}
\pi_{\Delta \shortto \Delta'} \vec{z}  = \proj{\Delta'} \proj{\Delta}^\dagger \vec{z}.
\end{equation}
We denote the $\Delta'$-marginal of a measure $\mu$ on $\mathbb{R}^{|\Delta|}$ by $\pi_{\Delta \shortto \Delta'} \pushfwd \mu$ because it is the measure on $\mathbb{R}^{|\Delta'|}$ defined by pushing $\mu$ forward by $\pi_{\Delta \shortto \Delta'}$. Precisely,
\begin{equation}
    \pi_{\Delta \shortto \Delta'} \pushfwd \mu(E) := \mu \left( \pi_{\Delta \shortto \Delta'}^{-1} E \right)
\end{equation}
for every Borel set $E \subseteq \mathbb{R}^{|\Delta'|}$, where $\pi_{\Delta \shortto \Delta'}^{-1} E$ is the preimage of $E$. Note that if $\mu_{\Delta}$ is a representing measure for a vector $\vec{y}_\Delta:=\riesz{\vec{y}}\left(  \dmon{ \vec{x}_{\Delta} }_{2\omega} \right)$ (as in \cref{e:local-moments}), then $\pi_{\Delta \shortto \Delta'} \pushfwd \mu_\Delta$ is a representing measure for the subvector $\vec{y}_{\Delta'}$.

%%%%%%%%%%%%%%%%%%%%%%%%%%%%%%%%%%%%%%%%%%
\section{Proof of \texorpdfstring{\Cref{thm:FF}}{Theorem \ref{thm:FF}}}\label{s:proof}
We now detail the proof of \Cref{thm:FF}. In a preliminary step (\cref{ss:nnz-clique-y}) we show that if $\vec{y}$ is nonzero and satisfies~\cref{e:lmis,,e:flat-moments,,e:flat-overlap}, then the clique subvectors $\vec{y}_{\Delta_1},\ldots,\vec{y}_{\Delta_m}$ defined in~\cref{e:local-moments} are also nonzero. This means $\rank \mommat^\omega_{\Delta_i}(\vec{y}) \geq 1$ for all $i \in [m]$, so we can introduce `local' atomic representing measures for the subvectors $\vec{y}_{\Delta_1},\ldots,\vec{y}_{\Delta_m}$. In a second step (\cref{ss:assembly}) we assemble these local measures to obtain the `global' representing measure in \Cref{thm:FF}.

%%%%%%%%%%%%%%%%%%%%%
\subsection{\texorpdfstring{The vectors $\vec{y}_{\Delta_i}$}{Clique subvectors} are not zero}\label{ss:nnz-clique-y}
We begin with a useful observation about a general \emph{dense} moment matrix
\begin{equation}
    \mommat^\omega(\tilde{\vec{y}}) := \riesz{\tilde{\vec{y}}}\big( \dmon{\tilde{\vec{x}}}_\omega \otimes \dmon{\tilde{\vec{x}}}_\omega \big),
\end{equation}
defined here for general vectors $\tilde{\vec{x}} \in \mathbb{R}^{\tilde{n}}$ and $\tilde{\vec{y}} \in \mathbb{R}^{\binom{\tilde{n}+2\omega}{\tilde{n}}}$. Later, we will apply this observation to the clique moment matrices $\mommat^\omega(\vec{y}_{\Delta_i}) = \mommat^\omega_{\Delta_i}(\vec{y})$, which are obtained with $\tilde{n}=|\Delta_i|$, $\tilde{\vec{x}}=\vec{x}_{\Delta_i}$, and $\tilde{\vec{y}} = \vec{y}_{\Delta_i}$.

\begin{lemma}\label{lem:zero-rank}
    Fix integers $\omega\geq d \geq 1$. Suppose $\tilde{\vec{y}} \in \mathbb{R}^{\binom{\tilde{n}+2\omega}{\tilde{n}}}$
    satisfies 
    $\mommat^\omega(\tilde{\vec{y}}) \succeq 0$,
    $\rank \mommat^{\omega - d}(\tilde{\vec{y}}) = \rank \mommat^\omega(\tilde{\vec{y}})$, and
    $\riesz{\tilde{\vec{y}}}(1) = 0$.
    Then $\tilde{\vec{y}}={0}$.
\end{lemma}

\begin{proof}
    If $\omega=d$, then $\rank \mommat^\omega(\tilde{\vec{y}}) = \rank \mommat^{0}(\tilde{\vec{y}}) = \riesz{\tilde{\vec{y}}}(1) = 0$, which implies that $\tilde{\vec{y}}=0$ by definition of $\mommat^\omega(\tilde{\vec{y}})$.

    Next, suppose $\omega > d$. For every $t\in\{0,\ldots,\omega\}$, the moment matrix $\mommat^t(\tilde{\vec{y}})$ is a principal submatrix of $\mommat^\omega(\tilde{\vec{y}})$. We claim that
    \begin{equation}\label{e:zero-rank-increase}
        \mommat^{t}(\tilde{\vec{y}}) = {0}
        \quad\implies\quad
        \mommat^{\lfloor \tfrac12 (t + \omega) \rfloor}(\tilde{\vec{y}}) = {0}.
    \end{equation}
    Indeed, since $\mommat^\omega(\tilde{\vec{y}})$ is positive semidefinite and $\mommat^{t}(\tilde{\vec{y}})$ is one of its principal submatrices,
    $\mommat^{t}(\tilde{\vec{y}}) = {0}$ implies that all rows of $\mommat^\omega(\tilde{\vec{y}})$ indexed by the monomials in $\dmon{\tilde{\vec{x}}}_t$ must vanish, that is,
    \begin{equation}
        \riesz{\tilde{\vec{y}}}\big( \dmon{\tilde{\vec{x}}}_t \otimes \dmon{\tilde{\vec{x}}}_\omega \big) = {0}.
    \end{equation}
    This is equivalent to $\riesz{\tilde{\vec{y}}}(\dmon{\tilde{\vec{x}}}_{t+\omega})={0}$. Since the matrix $\mommat^{\lfloor \frac12 (t + \omega) \rfloor}(\tilde{\vec{y}})$ depends only on (a subset of) this vector, it is the zero matrix.
    
    Now, $\mommat^0(\tilde{\vec{y}}) = \riesz{\tilde{\vec{y}}}(1) = 0$ by assumption. Since $\omega>d \geq 1$ by assumption, for all $t \in \{0,\ldots,\omega-d-1\}$ we have that $\lfloor \frac12 (t + \omega) \rfloor \geq t+1$. We can then repeatedly apply \cref{e:zero-rank-increase} to obtain that $\mommat^{\omega - d}(\tilde{\vec{y}}) = {0}$. Since $\rank \mommat^{\omega - d}(\tilde{\vec{y}}) = \rank \mommat^\omega(\tilde{\vec{y}})$ by assumption, we conclude that  $\mommat^{\omega}(\tilde{\vec{y}}) = {0}$. This implies that $\tilde{\vec{y}}={0}$.
\end{proof}

The next result states that if $\vec{y}\in\mathbb{R}^{|\SparseExponents{2\omega}|}$ is nonzero and satisfies the matrix inequalities in \cref{e:lmis} as well as the rank conditions in~\cref{e:flat-moments}, then all subvectors $\vec{y}_{\Delta_i}$ are nonzero.

\begin{proposition}\label{prop:nnz-clique-y}
    Fix $\omega \in \mathbb{N}$ and let $\vec{y}\in\mathbb{R}^{|\SparseExponents{2\omega}|}$ satisfy \cref{e:lmis,e:flat-moments}. There exists $i \in [m]$ such that $\vec{y}_{\Delta_i}={0}$ if and only if $\vec{y}={0}$.\end{proposition}

\begin{proof}
    The `if' part is obvious. For the converse, pick $i \in [m]$ such that $\smash{\vec{y}_{\Delta_i}={0}}$. The matrix $\smash{\mommat^\omega_{\Delta_i}(\vec{y})=\mommat^\omega(\vec{y}_{\Delta_i})}$ must then be zero and, since $\riesz{\vec{y}}(1)$ is one of its diagonal entries, we conclude that $\riesz{\vec{y}}(1) = 0$. 
    
    Next, for each $j \in [m]$ the vector $\smash{\vec{y}_{\Delta_j}}$ satisfies $\smash{\riesz{\vec{y}_{\Delta_j}}(1)=\riesz{\vec{y}}(1)=0}$. By \cref{e:lmis}, the dense moment matrix $\smash{\mommat^\omega(\vec{y}_{\Delta_j}) = \mommat^\omega_{\Delta_j}(\vec{y})}$ is positive semidefinite, and by \cref{e:flat-moments} it satisfies \mbox{$\smash{\rank \mommat^{\omega-d_j}(\vec{y}_{\Delta_j}) = \rank \mommat^\omega(\vec{y}_{\Delta_j})}$}. We then conclude from \Cref{lem:zero-rank} that $\vec{y}_{\Delta_j} = {0}$ for every $j \in [m]$, which in turn implies $\vec{y}={0}$.
\end{proof}

%%%%%%%%%%%%%%%%%%%%%%%%
\subsection{Existence of the representing measure}
\label{ss:assembly}

We now construct the atomic representing measure for $\vec{y}$ claimed in \Cref{thm:FF}, thereby proving that theorem.

Since $\vec{y}$ is nonzero by assumption, the clique subvectors $\vec{y}_{\Delta_1},\ldots,\vec{y}_{\Delta_m}$ are not zero by \Cref{prop:nnz-clique-y}. They also satisfy \cref{e:lmis,e:flat-moments} by assumption. We can then invoke standard results on truncated moment sequences (see, for example, \cite[Theorem~1.6]{CurtoFialkow2000} and \mbox{\cite[Theorem~1.6]{Laurent2005}}) to conclude that, for every clique index $i \in [m]$, the vector $\vec{y}_{\Delta_i}$ has a unique atomic representing measure supported on $r_i = \smash{\rank \mommat^\omega_{\Delta_i}(\vec{y})}\geq 1$ points in the set
\begin{equation}\label{e:feasible-sets}
    K_i := \left\{ \vec{z}\in \mathbb{R}^{|\Delta_i|}:\; \vec{g}_i(\vec{z}) \geq 0 \right\}.
\end{equation}
Specifically, there exist `clique atoms' $\vec{z}^{i}_{1},\ldots,\vec{z}^{i}_{r_i} \in K_i$ and scalars $ \lambda^{i}_{1},\ldots,\lambda^{i}_{r_i} > 0$ such that $\vec{y}_{\Delta_i}$ is represented by the atomic measure
\begin{equation}\label{e:local-measure}
    \mu_{\Delta_i} = 
    \lambda^{i}_{1} \delta_{ \vec{z}^{i}_{1} }
    + \cdots + 
    \lambda^{i}_{r_i} \delta_{ \vec{z}^{i}_{r_i} }.
\end{equation}
Next, we show that the measures $\mu_{\Delta_1},\ldots,\mu_{\Delta_m}$ have consistent marginals on particular clique intersections.

\begin{lemma}\label{lem:measure-prop}
    Under the assumptions of \Cref{thm:FF}, the measures $\mu_{\Delta_1},\ldots,\mu_{\Delta_m}$ satisfy the following consistency condition:
    for every $i \in \{2,\ldots,m\}$, there exists $j \in [i-1]$ satisfying \cref{e:rip} such that
    \begin{equation}\label{e:marginals}
    \pi_{\Delta_i \shortto \Delta_i\cap\Delta_j}  \pushfwd \mu_{\Delta_i}
    =
    \pi_{\Delta_j \shortto \Delta_i\cap\Delta_j}  \pushfwd \mu_{\Delta_j}.
    \end{equation}
\end{lemma}

\begin{proof}
Fix $j \in [i-1]$ such that \cref{e:rip} and \cref{e:flat-overlap} hold, which exists by the assumptions of \Cref{thm:FF}. The marginals $\pi_{\Delta_i \shortto \Delta_i\cap\Delta_j}  \pushfwd \mu_{\Delta_i}$ and $\pi_{\Delta_j \shortto \Delta_i\cap\Delta_j}  \pushfwd \mu_{\Delta_j}$ are representing measures for $\vec{y}_{\Delta_i \cap \Delta_j}$ because this is a subvector of both $\vec{y}_{\Delta_i}$ and $\vec{y}_{\Delta_j}$, which are represented by $\mu_{\Delta_i}$ and $\mu_{\Delta_j}$. However, by a classical flat extension theorem for moment vectors (e.g., \mbox{\cite[Theorem~5.19]{Laurent2009}}), the rank condition in \cref{e:flat-overlap} implies that $\vec{y}_{\Delta_i \cap \Delta_j}$ has a \emph{unique} representing measure. Identity \cref{e:marginals} follows.
\end{proof}

We now employ a construction of Lasserre's \cite[Lemma~6.4]{Lasserre2006}, which states that if the cliques $\Delta_1,\ldots,\Delta_m$ satisfy \cref{e:rip}, then for any measures $\mu_{\Delta_1},\ldots,\mu_{\Delta_m}$ satisfying the consistency condition in \Cref{lem:measure-prop} there exists a measure $\mu$ supported on the set $K$ from \cref{e:bsa} with $\mu_{\Delta_i}$ as its ${\Delta_i}$-marginal.\footnote{The original statement is for probability measures and assumes that \cref{e:marginals} holds for all pairs of intersecting cliques, but its proof actually uses only the consistency condition in \Cref{lem:measure-prop}.}
The measure $\mu$ obtained by applying this result to the measures $\mu_{\Delta_1},\ldots,\mu_{\Delta_m}$ in \cref{e:local-measure} is clearly a representing measure for $\vec{y}$. It then follows from general arguments (see, e.g., \cite[Satz~4]{Richter1957} and \cite[Theorem~5.8]{Laurent2009}) that  $\vec{y}$ also has a finitely atomic representing measure with atoms in $K$. In fact, as demonstrated in \cref{ss:explicit-recovery}, the measure $\mu$ obtained with Lasserre's construction is already atomic. Finally, the number $r$ of atoms in $\mu$ satisfies
\begin{equation}
    r \geq \max_{i \in [m]} \,\rank \mommat^\omega_{\Delta_i}(\vec{y})
\end{equation}
because, for each $i \in [m]$, the clique marginal $\mu_{\Delta_i}$ must have exactly $\rank \mommat^\omega_{\Delta_i}(\vec{y})$ atoms, which must be projections onto clique $\Delta_i$ of the atoms of $\mu$. \Cref{thm:FF} is therefore proved.

%%%%%%%%%%%%%%%%%%%%%%%%%%%%%%%%%%%%%%%%%%%%%%%%%%%%%%%%%%%%%%%%
\section{Recovering representing measures}\label{s:recovery}
We now replicate the measure assembly construction from \cite[Lemma~6.4]{Lasserre2006} in the special case of atomic measures. We do so to explicitly demonstrate that if $\vec{y}$ satisfies the conditions of \Cref{thm:FF}, then one can algorithmically recover an atomic representing measure for it (see \Cref{alg:measure_recovery} for a pseudo-code constructing the atoms and weights of such a measure). This is useful in practice because, when $\vec{y}$ is the optimal solution of the moment relaxation \cref{e:mom}, the atoms of the representing measure are optimizers of the corresponding polynomial optimization problem \cref{e:pop} (cf. \Cref{thm:FF-POP}). We also show in \cref{ss:max-support} that the recovered measure has maximal support among all possible atomic representing measures. Finally, in \cref{ss:other-measures} we use this observation and convex programming to search for other atomic representing measures.

%%%%%%%%%%%%%%%%%%%%%%%%%%%%%%%%%%%%%
%%% Algorithm summary
\begin{algorithm}[t]
\caption{Construction of an atomic representing measure}
\label{alg:measure_recovery}
\begin{algorithmic}[1]
\Require
Cliques $\Delta_1,\ldots,\Delta_m$ ($m\geq 2$)

\Require
A vector $\vec{y}=(y_\alpha)_{\alpha \in \SparseExponents{2\omega}}$ satisfying the conditions in \Cref{thm:FF}

\Ensure
Atoms and weights in an atomic representing measure for $\vec{y}$

\For{$i=1,\dots,m$}
    \State Extract the clique subvector $\vec{y}_{\Delta_i}$
    \State $\{(\vec{z}^i_\ell, \lambda^i_\ell)\}_{\ell = 1}^{r_i} \gets $ Atoms \& weights of a measure representing $\vec{y}_{\Delta_i}$ \cite{Henrion2005extraction}
\EndFor

\State Initialize $s = r_1$, $\{(\vec{u}_k, \nu_k)\}_{k=1}^s\gets \{(\vec{z}^1_\ell, \lambda^1_\ell)\}_{\ell=1}^s$, and $U_1 \gets \Delta_1$

\For{$i=2,\dots,m$}
    \State $U_i \gets U_{i-1}\cup \Delta_i$
    \State Choose $j<i$ such that $\Delta_i\cap U_{i-1}\subseteq \Delta_j$ (existence guaranteed by RIP)
    \State Extract the clique subvector $\vec{y}_{\Delta_i \cap \Delta_j}$
    \State $\{(\vec{v}_\tau,\theta_\tau)\}_{\tau=1}^t \gets $ Atoms \& weights of a measure representing $\vec{y}_{\Delta_i \cap \Delta_j}$ \cite{Henrion2005extraction}
    \State Initialize \texttt{atoms} $\gets \{\}$ and \texttt{weights} $\gets \{\}$
    \For{$\tau=1,\dots,t$}
        \For{$k = 1,\ldots,s$ such that $\pi_{U_{i-1} \shortto \Delta_i \cap \Delta_j} \vec{u}_k = \vec{v}_\tau$}
            \For{$\ell = 1,\ldots, r_i$ such that $\pi_{\Delta_{i} \shortto \Delta_i \cap \Delta_j} \vec{z}^{i}_\ell = \vec{v}_\tau$}
                \State $\vec{w}_{\tau,k,\ell} \gets$ concatenate $\vec{u}_k$ and $\vec{z}^{i}_\ell$ to satisfy \cref{e:atom-concatenation}
                \State \texttt{atoms} $\gets$ append $\vec{w}_{\tau,k,\ell}$
                \State \texttt{weights} $\gets$ append $\nu_k\,\lambda^i_\ell / \theta_\tau$
            \EndFor
        \EndFor
    \EndFor

    \State $s \gets$ number of elements in \texttt{atoms}
    \State $\{(\vec{u}_k , \nu_k)\}_{k=1}^s\gets$ enumerate \texttt{atoms} and corresponding \texttt{weights}
\EndFor

\State \Return \texttt{atoms} and \texttt{weights}
\end{algorithmic}
\end{algorithm}
%%%%%%%%%%%%%%%%%%%%%%%%%%%%%%%%%%%%%

%%%%%%%%%%%%%%%%
\subsection{Construction of an atomic representing measure}
\label{ss:explicit-recovery}

Let $\vec{y}=(y_\alpha)_{\alpha \in \SparseExponents{2\omega}}$ satisfy the conditions in \Cref{thm:FF}.
For each clique $\Delta_i$, $i \in [m]$, an atomic measure $\mu_{\Delta_i}$ of the form \cref{e:local-measure} can be recovered from the vector $\vec{y}_{\Delta_i}$ using an algorithm based on standard linear algebra operations~\cite{Henrion2005extraction}. These measures can be assembled using an inductive procedure that, at each step $i\in[m]$, constructs an atomic measure $\mu^{i}$  on $\mathbb{R}^{|\Delta_1 \cup \cdots \cup \Delta_i|}$ whose marginals on the cliques $\Delta_1,\ldots,\Delta_i$ are precisely $\mu_{\Delta_1},\ldots,\mu_{\Delta_i}$. The atomic measure $\mu^{m}$ then represents the vector $\vec{y}$.

We initialize $\mu^{1}=\mu_{\Delta_1}$ and, for notational simplicity, set $U_i = \Delta_1 \cup \cdots \cup \Delta_i$. Then, for every $i \in \{2,\ldots,m\}$, suppose we have an atomic measure
\begin{equation}
    \mu^{i-1} = \nu_1 \delta_{\vec{u}_1} + \cdots + \nu_s \delta_{\vec{u}_s}
\end{equation}
with weights $\nu_1,\ldots,\nu_s>0$ and points $\smash{\vec{u_1},\ldots,\vec{u}_s \in \mathbb{R}^{|U_{i-1}|}}$, whose marginals on the cliques $\Delta_1,\ldots,\Delta_{i-1}$ are precisely $\smash{\mu_{\Delta_1},\ldots,\mu_{\Delta_{i-1}}}$. We need to construct an atomic measure $\mu^{i}$ on $\mathbb{R}^{|U_{i}|}$ with $\mu^{i-1}$ as its $U_{i-1}$-marginal and $\mu_{\Delta_i}$ as its $\Delta_i$-marginal. 

To accomplish this, let us use \Cref{lem:measure-prop} to choose $j \in [i-1]$ for which \cref{e:marginals} holds. 
We claim that the $\Delta_i \cap \Delta_j$-marginals of $\mu^{i-1}$ and $\mu_{\Delta_i}$ match.
Indeed, since $\Delta_i \cap \Delta_j \subseteq \Delta_j \subseteq U_{i-1}$, and since the $\Delta_j$-marginal of $\mu^{i-1}$ is $\mu_{\Delta_j}$ by the induction assumption, we have 
\begin{equation}
\begin{aligned}
    \pi_{U_{i-1} \shortto \Delta_i \cap \Delta_j} \pushfwd \mu^{i-1} 
    &\;\,\,=\;\; \left( \pi_{\Delta_j \shortto \Delta_i \cap \Delta_j} \circ \pi_{U_{i-1} \shortto \Delta_j} \right) \pushfwd \mu^{i-1} 
    \\
    &\overset{\phantom{\text{\cref{e:marginals}}}}{=} \pi_{\Delta_j \shortto \Delta_i \cap \Delta_j} \pushfwd \left( \pi_{U_{i-1} \shortto \Delta_j} \pushfwd \mu^{i-1}  \right)
    \\
    &\overset{\phantom{\text{\cref{e:marginals}}}}{=} \pi_{\Delta_j \shortto \Delta_i \cap \Delta_j} \pushfwd \mu_{\Delta_j}
    \\
    &\overset{\text{\cref{e:marginals}}}{=}
    \pi_{\Delta_i \shortto \Delta_i \cap \Delta_j} \pushfwd \mu_{\Delta_i}.
\end{aligned}
\end{equation}

Since the marginals of atomic measures are atomic, this equation means that there exist $t\geq 1$ points $\vec{v}_1,\ldots,\vec{v}_t\in \mathbb{R}^{|\Delta_i \cap \Delta_j|}$ and scalars $\theta_1,\ldots,\theta_t>0$ such that
\begin{align}\label{e:matching-measure}
    \pi_{U_{i-1} \shortto \Delta_i \cap \Delta_j} \pushfwd \mu^{i-1} 
    = \theta_1 \delta_{\vec{v}_1} + \cdots + \theta_t \delta_{\vec{v}_t} =
    \pi_{\Delta_i \shortto \Delta_i \cap \Delta_j} \pushfwd \mu_{\Delta_i}.
\end{align}
In fact, $\theta_1 \delta_{\vec{v}_1} + \cdots + \theta_t \delta_{\vec{v}_t}$ is the unique representing measure for the vector $\vec{y}_{\Delta_i \cap \Delta_j}$ used in the proof of \Cref{lem:measure-prop}.

On the other hand, by definition of $\mu^{i-1}$ and of $\mu_{\Delta_i}$, we have
\begin{subequations}
    \begin{align}
    \pi_{U_{i-1} \shortto \Delta_i \cap \Delta_j} \pushfwd \mu^{i-1}  
    &= \sum_{k=1}^s \nu_k \delta_{\pi_{U_{i-1} \shortto \Delta_i \cap \Delta_j} \vec{u}_k} ,
    \\
    \pi_{\Delta_i \shortto \Delta_i \cap \Delta_j} 
    \pushfwd \mu_{\Delta_i}
    &= \sum_{\ell=1}^{r_i} \lambda^i_\ell \delta_{\pi_{\Delta_{i} \shortto \Delta_i \cap \Delta_j} \vec{z}^i_\ell}.
\end{align}
\end{subequations}
Substituting these expressions into \cref{e:matching-measure} and noting the uniqueness of the representing measure in $\Delta_i\cap\Delta_j$ leads to two key observations. First, for every $\tau \in [t]$, the index sets
\begin{subequations}\label{e:matching-proj}
    \begin{align}
        \mathcal{A}_\tau &:= 
        \big\{ 
        k \in [s]:\;  \pi_{U_{i-1} \shortto \Delta_i \cap \Delta_j} \vec{u}_k = \vec{v}_\tau 
        \big\}
        \\
        \mathcal{B}_\tau &:= 
        \big\{ 
        \ell \in [r_i]:\;  \pi_{\Delta_{i} \shortto \Delta_i \cap \Delta_j} \vec{z}^{i}_\ell = \vec{v}_\tau
        \big\}
    \end{align}
\end{subequations}
are non-empty, pairwise disjoint, and cover $[s]$ and $[r_i]$, respectively. Second,
\begin{equation}\label{e:nu-lambda-coco}
     \sum_{ k \in \mathcal{A}_\tau } \nu_k  = 
     \sum_{ \ell \in \mathcal{B}_\tau } \lambda^i_\ell =
     \theta_\tau.
\end{equation}

Now, thanks to \cref{e:matching-proj}, for every $\tau \in [t]$ and every pair $(k,\ell)\in\mathcal{A}_\tau \times \mathcal{B}_\tau$ 
the entries of $\vec{u}_k$ and $\vec{z}^i_\ell$ indexed by the clique overlap $\Delta_i \cap \Delta_j$ match. Consequently, we can `concatenate' $\vec{u}_k$ and $\vec{z}^i_\ell$ to construct a point $\vec{w}_{\tau,k,\ell} \in \mathbb{R}^{|U_i|}$ satisfying
\begin{equation}\label{e:atom-concatenation}
    \pi_{U_i \shortto U_{i-1}} \vec{w}_{\tau,k,\ell} = \vec{u}_k
    \qquad\text{and}\qquad
    \pi_{U_i \shortto \Delta_i} \vec{w}_{\tau,k,\ell} = \vec{z}^i_\ell.
\end{equation}
These identities, together with \cref{e:nu-lambda-coco}, enable one to verify that the atomic measure
\begin{equation}\label{e:atomic-lasserre}
    \mu^i := \sum_{\tau=1}^t \sum_{k \in \mathcal{A}_\tau} \sum_{\ell \in \mathcal{B}_\tau} \left(\frac{\nu_k \lambda^i_\ell}{\theta_\tau} \right) \delta_{\vec{w}_{\tau,k,\ell}}
\end{equation}
has $\mu^{i-1}$ and $\mu_{\Delta_i}$ as its $U_{i-1}$- and $\Delta_i$-marginals, as desired. This concludes the construction; we remark only that \cref{e:atomic-lasserre} is the atomic version of~\cite[Eq.~(6.7)]{Lasserre2006}.

%%%%%%%%%%%%%%%%
\subsection{Maximality of the support}\label{ss:max-support}
We now prove that, among all possible atomic representing measures for a vector $\vec{y}$ satisfying the conditions of \Cref{thm:FF}, the one constructed in \cref{ss:explicit-recovery} has maximal support. In particular, this reveals that the atoms of the representing measure do not depend on how the cliques are ordered during the construction, provided of course that the ordering  satisfies \cref{e:rip}. As usual, we write $\supp(\mu)$ for the support of a measure $\mu$.

\begin{proposition}\label{prop:supp-maximality}
    Let $\vec{y}$ satisfy the conditions of \Cref{thm:FF} and let $\mu$ be its atomic representing measure constructed with the procedure of \cref{ss:explicit-recovery}. If $\mu'$ is any other atomic representing measure for $\vec{y}$, then  $\supp(\mu') \subseteq \supp(\mu)$.
\end{proposition}

\begin{remark}
    In the algebraic language of core varieties \cite{Fialkow2017corevariety,Blekherman2020corevariety}, this result says that the procedure of \cref{ss:explicit-recovery} constructs the core variety of the Riesz functional $\riesz{\vec{y}}$.
\end{remark}

\begin{proof}
    Recall that, for each $i\in[m]$, the atomic measure $\mu_{\Delta_i}$ in \cref{e:local-measure} is unique and has finite support on the clique atoms, i.e., $\supp(\mu_{\Delta_i}) = \{\vec{z}^i_1,\ldots,\vec{z}^i_{r_i}\}$. Consider the set
    \begin{equation}
        \mathcal{X} = \left\{ \vec{x}\in\mathbb{R}^n:\; \proj{\Delta_i}\vec{x} \in \supp(\mu_{\Delta_i}) \quad \forall i \in [m]\right\},
    \end{equation}
    the definition of which does not depend on how the cliques are ordered. Since $\mu$ and $\mu'$ have $\mu_{\Delta_1},\ldots,\mu_{\Delta_m}$ as their clique marginals, they must both be supported on a subset of $\mathcal{X}$. Then, the claim follows if we can show that $\mathcal{X} \subseteq  \supp(\mu)$, which also proves that $\supp(\mu) = \mathcal{X}$ is independent of the clique ordering.

    With this in mind, fix $\vec{x}\in\mathcal{X}$. As in \cref{ss:explicit-recovery}, set $U_i=\Delta_1\cup \cdots \cup \Delta_i$ for every $i\in[m]$. To show that $\vec{x} \in \supp(\mu)$ it suffices to verify that, for every $i\in[m]$, the projection $\proj{U_i}\vec{x}$ is an atom of the measure $\mu^i$ in \cref{e:atomic-lasserre}. This is true for $i=1$ because $\mu^1 = \mu_{\Delta_1}$ and $\proj{U_1}\vec{x} = \proj{\Delta_1}\vec{x} \in \supp(\mu_{\Delta_1})$ by definition of $\vec{x}$. For $i \in \{2,\ldots,m\}$, we proceed by induction and claim that if $\proj{U_{i-1}}\vec{x}$ is an atom of $\mu^{i-1}$, then $\proj{U_{i}}\vec{x}$ is an atom of $\mu^{i}$.
    
    To prove this claim, recall from \cref{ss:explicit-recovery} that the atoms of $\mu^i$ are constructed by concatenating every atom $\vec{u}_k$ of $\mu^{i-1}$ with all atoms $\vec{z}^i_\ell$ of $\mu_{\Delta_i}$ such that
    \begin{equation}
        \pi_{U_{i-1}\shortto U_{i-1} \cap \Delta_i} \vec{u}_k = \pi_{\Delta_i\shortto U_{i-1} \cap \Delta_i} \vec{z}^i_\ell.
    \end{equation}
    We now claim that this identity holds for $\vec{u}_k=\proj{U_{i-1}}\vec{x}$ and $\vec{z}^i_\ell = \proj{\Delta_i} \vec{x}$, which are valid choices by the induction assumption and the definition of $\vec{x}$. Indeed, by \cref{e:rip}, there exists a clique index $j\in[i-1]$ such that $U_{i-1} \cap \Delta_i \subseteq \Delta_j$ and, consequently,
    \begin{align}
        \pi_{U_{i-1}\shortto U_{i-1} \cap \Delta_i} \proj{U_{i-1}}\vec{x}
        =\pi_{U_{i-1}\shortto \Delta_j \cap \Delta_i} \proj{U_{i-1}}\vec{x}
        = \proj{\Delta_i \cap \Delta_j}\vec{x}
        =\pi_{\Delta_i\shortto \Delta_j \cap \Delta_i} \proj{\Delta_i}\vec{x}.
    \end{align}
    Since concatenating $\proj{U_{i-1}}\vec{x}$ with $\proj{\Delta_i} \vec{x}$ yields exactly $\proj{U_i}\vec{x}$ we conclude that this point is an atom of $\mu^i$, as desired. This concludes the proof.
\end{proof}

%%%%%%%%%%%%%%%%
\subsection{Finding other atomic representing measures}\label{ss:other-measures}
Atomic representing measures with the largest possible support are desirable in the context of moment relaxations of a POP because they allow one to maximize the number of distinct minimizers for the POP that one recovers. Other times, however, one may want atomic representing measures that optimize different criteria.

Alternatives to the maximal representing measure $\mu$ constructed in \cref{ss:explicit-recovery}, when they exist, can be easily found using computational tools for convex optimization.
Specifically, let $\vec{x}_1,\ldots,\vec{x}_r$ be the (known) atoms of $\mu$ and fix any convex lower-semicontinuous function $c:\mathbb{R}^r \to \mathbb{R}$. 
One can then solve the convex program
\begin{equation}\label{e:lp}
    \begin{aligned}
        \min_{\gamma_1,\ldots,\gamma_r} \quad &
        c(\gamma_1,\cdots, \gamma_r)
        \\
        \text{s.t.} \quad 
        &\gamma_1 \smon{\vec{x}_1}_{2\omega} + \cdots + \gamma_r \smon{\vec{x}_r}_{2\omega} = \vec{y} \\
        &\gamma_1,\ldots,\gamma_r \geq 0
    \end{aligned}
\end{equation}
and obtain an atomic representing measure $\gamma_1 \delta_{\vec{x}_1} + \cdots + \gamma_r \delta_{\vec{x}_r}$ for $\vec{y}$. Note that the strict feasibility of this problem follows from the construction of \cref{ss:explicit-recovery}. Note also that the feasible set is compact because the constraints include the equation $\gamma_1 + \cdots + \gamma_r = y_{{0}}$, so the minimum is attained.

It is clear that varying $c$ should lead to different representing measures in general, although it may not be easy to design a cost function to single out a representing measure with desired properties (e.g., with the minimum number of atoms). Representing measures that are \emph{extreme} in the sense of convex analysis, however, can be found simply by taking $c$ to be a linear cost function. In fact, in this case \cref{e:lp} is a standard-form linear program and an extreme optimal point, corresponding to an extreme atomic representing measure, can be found with the simplex method. Since extreme solutions to linear programs are often sparse, this approach can help one find atomic representing measures with small support.

\section{Examples}\label{s:examples}

We conclude by giving some examples to illustrate our theoretical results.

\subsection{A first example}
\label{ss:ex1}
We consider a correlatively sparse truncated moment problem on the whole space $K=\mathbb{R}^n$ with $n=3$ variables and two cliques $\Delta_1=\{1,2\}$, $\Delta_2=\{2,3\}$, which clearly satisfy \cref{e:rip}. We choose $\omega=2$ as the relaxation order and consider the vector $\vec{y}\in\mathbb{R}^{|\SparseExponents{4}|}$ whose $25$ entries are
\begin{equation}
    \begin{gathered}
    y_{000} = 1, \quad
    y_{100} = 0, \quad
    y_{010} = 0, \quad
    y_{001} = 0, \quad
    y_{200} = 1, \quad
    y_{110} = 0, \quad
    y_{020} = 0, \\
    y_{011} = 0, \quad
    y_{002} = 1, \quad
    y_{300} = 0, \quad
    y_{210} = 0, \quad
    y_{120} = 0, \quad
    y_{030} = 0, \quad
    y_{021} = 0, \\
    y_{012} = 0, \quad
    y_{003} = 0, \quad
    y_{400} = 1, \quad
    y_{310} = 0, \quad
    y_{220} = 0, \quad
    y_{130} = 0, \quad
    y_{040} = 0, \\ 
    y_{031} = 0, \quad
    y_{022} = 0, \quad
    y_{013} = 0, \quad
    y_{004} = 1.
    \end{gathered}
\end{equation}
The moment matrices appearing in \cref{e:lmis} are then
\begin{equation}
    \mommat^\omega_{\Delta_1}(\vec{y}) = 
    \begin{bmatrix}
         1 & 0 & 0 & 1 & 0 & 0\\
         0 & 1 & 0 & 0 & 0 & 0\\
         0 & 0 & 0 & 0 & 0 & 0\\
         1 & 0 & 0 & 1 & 0 & 0\\
         0 & 0 & 0 & 0 & 0 & 0\\
         0 & 0 & 0 & 0 & 0 & 0
    \end{bmatrix}
    \quad\text{and}\quad
    \mommat^\omega_{\Delta_2}(\vec{y}) = 
    \begin{bmatrix}
        1 & 0 & 0 & 0 & 0 & 1\\
        0 & 0 & 0 & 0 & 0 & 0\\
        0 & 0 & 1 & 0 & 0 & 0\\
        0 & 0 & 0 & 0 & 0 & 0\\
        0 & 0 & 0 & 0 & 0 & 0\\
        1 & 0 & 0 & 0 & 0 & 1
    \end{bmatrix}.
\end{equation}
They are easily seen to be positive semidefinite.
No localizing matrices need be considered because $K=\mathbb{R}^3$ is the full space in this example. This also means we have $d_1=d_2=1$ in \cref{e:di-def}, so the moment submatrices appearing in \cref{e:flat-moments} are
\begin{equation}
    \mommat^{\omega-d_1}_{\Delta_1}(\vec{y}) = 
    \begin{bmatrix}
         1 & 0 & 0\\
         0 & 1 & 0\\
         0 & 0 & 0
    \end{bmatrix}
    \quad\text{and}\quad
    \mommat^{\omega-d_2}_{\Delta_2}(\vec{y}) = 
    \begin{bmatrix}
        1 & 0 & 0\\
        0 & 0 & 0\\
        0 & 0 & 1
    \end{bmatrix}.
\end{equation}
Finally, the `overlap' moment matrices appearing in \cref{e:flat-overlap} are
\begin{equation}
    \mommat^\omega_{\Delta_1 \cap \Delta_2}(\vec{y}) = 
    \begin{bmatrix}
        1 & 0 & 0\\
        0 & 0 & 0\\
        0 & 0 & 0
    \end{bmatrix}
    \quad\text{and}\quad
    \mommat^{\omega-1}_{\Delta_1 \cap \Delta_2}(\vec{y}) = 
    \begin{bmatrix}
        1 & 0\\
        0 & 0
    \end{bmatrix}.
\end{equation}

With all these matrices at hand, one can easily verify that the conditions of \Cref{thm:FF} hold with
\begin{subequations}\label{e:ex1-ranks}
    \begin{gather}
        \rank\mommat_{\Delta_{1}}^{\omega}(\vec{y}) =
        \rank\mommat_{\Delta_{1}}^{\omega-d_{1}}(\vec{y}) = 2,
        \\
        \rank\mommat_{\Delta_{2}}^{\omega}(\vec{y}) =
        \rank\mommat_{\Delta_{2}}^{\omega-d_{2}}(\vec{y}) = 2,
        \\
        \rank\mommat_{\Delta_{1}\cap \Delta_{2}}^{\omega}(\vec{y}) =
        \rank\mommat_{\Delta_{1}\cap \Delta_{2}}^{\omega-1}(\vec{y}) = 1.
    \end{gather}
\end{subequations}
We conclude that the vector $\vec{y}$ has an atomic representing measure. 
To recover it, we first apply the algorithm described in~\cite{Henrion2005extraction} to obtain atomic representing measures for the vectors $\vec{y}_{\Delta_1}$ and $\vec{y}_{\Delta_2}$ from the moment matrices $\mommat_{\Delta_{1}}^{\omega}(\vec{y})$ and $\mommat_{\Delta_{2}}^{\omega}(\vec{y})$, yielding
\begin{equation}
    \mu_{\Delta_1} = \tfrac12 \delta_{(1,0)} + \tfrac12 \delta_{(-1,0)}
    \quad\text{and}\quad
    \mu_{\Delta_2} = \tfrac12 \delta_{(0,1)} + \tfrac12 \delta_{(0,-1)}.
\end{equation}
Then, we run the second part of \cref{alg:measure_recovery} to concatenate the atoms $\vec{z}^{1}_{1}=(1,0)$  and  $\vec{z}^{1}_{2}=(-1,0)$ for clique $\Delta_1$ with the atoms $\vec{z}^{2}_{1}=(0,1)$ and $\vec{z}^{2}_{2}=(0,-1)$ for clique $\Delta_2$. The procedure yields four `global' atoms, $\pm (1,0,1)$ and $\pm(1,0,-1)$, together with the associated weights in the atomic representing measure
\begin{equation}
    \mu = \tfrac14 \delta_{(1,0,1)} + \tfrac14 \delta_{(1,0,-1)} + \tfrac14 \delta_{(-1,0,1)} + \tfrac14 \delta_{(-1,0,-1)}.
\end{equation}
Indeed, one can easily verify that
\begin{equation}\label{e:ex1-atomic-dec}
    \vec{y} = 
    \tfrac14 \smon{(1,0,1)}_{4} + 
    \tfrac14 \smon{(1,0,-1)}_{4} + 
    \tfrac14 \smon{(-1,0,1)}_{4} + 
    \tfrac14 \smon{(-1,0,-1)}_{4}.
\end{equation}

This atomic decomposition has maximal support by \Cref{prop:supp-maximality}. However, it is far from unique: indeed, $\vec{y}$ is represented by the atomic measure $\mu_\lambda$ in \cref{e:mu-lambda} for every $\lambda \in [0,\frac12]$. The two extreme measures within this family, $\mu_0$ and $\smash{\mu_{1/2}}$, can be found by solving the convex program \cref{e:lp} with the linear cost functions $c(\gamma_1,\ldots,\gamma_4)=\gamma_1$ and $c(\gamma_1,\ldots,\gamma_4)=\gamma_2$, respectively. In fact, any linear function $c_1\gamma_1 + \cdots + c_4 \gamma_4$ is uniquely minimized by $\mu_0$ if $c_1+c_4 > c_2+c_3$ and by $\smash{\mu_{1/2}}$ if $c_1+c_4 < c_2+c_3$. If $c_1+c_4 = c_2+c_3$, all measures $\mu_\lambda$ with $\lambda \in [0,\frac12]$ are optimal.

Finally, the existence of a representing measure for the vector $\vec{y}$ in this example could also be determined using \cite[Theorem~3.7]{Lasserre2006}, but not using \cite[Theorem~3.5]{Nie2024} because not all matrix ranks in \cref{e:ex1-ranks} are equal. Next, we give an example where neither of these results applies, demonstrating the usefulness of \Cref{thm:FF}.

%%%%%%%%%%%%%%%%%%%%%%%%%%%%%%%%%%%%%%%%%%
\subsection{Minimizers of a POP}
\label{ss:ex-pop}

We now use \Cref{thm:FF-POP} to detect the finite convergence of sparse moment relaxations for a simple POP with $n=4$ variables,
\begin{equation}\label{e:ex-pop}
    \begin{aligned}
    f^* = \min_{\vec{x}\in\mathbb{R}^{4}}\quad & \left(x_{1}^2-1\right)^{2}+\left(x_{2}^{2}-1\right)^{2}+x_3^2+\left(x_{4}^{2}-1\right)^{2}\\
    \text{s.t.}\quad 
     & g_1(\vec{x}_{\Delta_1}) := 3-x_{1}^{2}-x_{2}^{2}\geq0 \\
     & g_2(\vec{x}_{\Delta_2}) := 3-x_{2}^{2}-x_{3}^{2}\geq0 \\
     & g_3(\vec{x}_{\Delta_3}) := 3-x_{3}^{2}-x_{4}^{2}\geq0.
    \end{aligned}
\end{equation}
The optimal value $f^*=0$ is attained by the minimizers
\begin{equation}
    \begin{gathered}
        \mathbf{x}^*_1 = (\phantom{-}1,\phantom{-}1,\phantom{-}0,\phantom{-}1), 
        \quad
        \mathbf{x}^*_2 = (\phantom{-}1,\phantom{-}1,\phantom{-}0,-1),
        \\
        \mathbf{x}^*_3 = (\phantom{-}1,-1,\phantom{-}0,\phantom{-}1), 
        \quad
        \mathbf{x}^*_4 = (\phantom{-}1,-1,\phantom{-}0,-1),
        \\
        \mathbf{x}^*_5 = (-1,-1,\phantom{-}0,\phantom{-}1), 
        \quad
        \mathbf{x}^*_6 = (-1,-1,\phantom{-}0,-1),
        \\
        \mathbf{x}^*_7 = (-1,\phantom{-}1,\phantom{-}0,\phantom{-}1), 
        \quad
        \mathbf{x}^*_8 = (-1,\phantom{-}1,\phantom{-}0,-1).
    \end{gathered}
\end{equation}

\subsubsection{Solution of the moment relaxation}
We used MOSEK (version 10.1.24) \cite{mosek} to solve the sparse relaxation \cref{e:mom} for increasing relaxation orders $\omega$, which we built using the variable cliques $\Delta_{1}=\{1,2\}$, $\Delta_{2}=\{2,3\}$ and $\Delta_{3}=\{3,4\}$. These are easily seen to satisfy \cref{e:rip}. The relaxations have the sharp optimal value $f^*_{\omega} = 0$ at all orders $\omega \geq 2$ (the smallest relaxation order) because the objective polynomial in \cref{e:ex-pop} is a sum of univariate squares.

This finite convergence can be detected using \cref{thm:FF-POP} at $\omega=3$. For this relaxation order, the optimal vector $\vec{y}$ has 70 entries that we do not report for brevity, all of which are zero or one up to four decimal places. Rounding this numerical solution yields a binary vector $\vec{y}$ that is optimal for \cref{e:mom}: its cost value is zero, which coincides with the trivial lower bound on the optimal value of  \cref{e:mom}, 
and it is feasible because the corresponding moment matrices $\smash{\mommat_{\Delta_i}^\omega(\vec{y})}$ and localizing matrices $\smash{\mommat^\omega_{\Delta_{i}}(g_i\vec{y})}$ are
\begin{subequations}
    \begin{align}
    \mommat_{\Delta_1}^\omega(\vec{y})
    &=
    \scalebox{0.90}{$
    \begin{bNiceMatrix}[c,margin]
    \CodeBefore
    \rectanglecolor{black!10}{1-1}{3-3}
    \rectanglecolor{black!10}{1-5}{3-6}
    \rectanglecolor{black!10}{1-8}{3-8}
    \rectanglecolor{black!10}{5-8}{6-8}
    \rectanglecolor{black!10}{5-1}{6-3}
    \rectanglecolor{black!10}{8-1}{8-3}
    \rectanglecolor{black!10}{8-5}{8-6}
    \rectanglecolor{black!10}{5-5}{6-6}
    \rectanglecolor{black!10}{8-8}{8-8}
    \Body
    1 & 0 & 1 & 0 & 0 & 0 & 0 & 1 & 0 & 0\\
    0 & 1 & 0 & 1 & 0 & 0 & 0 & 0 & 1 & 0\\
    1 & 0 & 1 & 0 & 0 & 0 & 0 & 1 & 0 & 0\\
    0 & 1 & 0 & 1 & 0 & 0 & 0 & 0 & 1 & 0\\
    0 & 0 & 0 & 0 & 1 & 0 & 1 & 0 & 0 & 1\\
    0 & 0 & 0 & 0 & 0 & 1 & 0 & 0 & 0 & 0\\
    0 & 0 & 0 & 0 & 1 & 0 & 1 & 0 & 0 & 1\\
    1 & 0 & 1 & 0 & 0 & 0 & 0 & 1 & 0 & 0\\
    0 & 1 & 0 & 1 & 0 & 0 & 0 & 0 & 1 & 0\\
    0 & 0 & 0 & 0 & 1 & 0 & 1 & 0 & 0 & 1
    \end{bNiceMatrix}$},
    &
    \mommat^\omega_{\Delta_{1}}(g_1\vec{y})
    &= \scalebox{0.90}{$
    \begin{bmatrix}
    1 & 0 & 1 & 0 & 0 & 1\\
    0 & 1 & 0 & 0 & 0 & 0\\
    1 & 0 & 1 & 0 & 0 & 1\\
    0 & 0 & 0 & 1 & 0 & 0\\
    0 & 0 & 0 & 0 & 1 & 0\\
    1 & 0 & 1 & 0 & 0 & 1
    \end{bmatrix}$},
    \end{align}
    \begin{align}
    \mommat_{\Delta_2}^\omega(\vec{y})
    &=\scalebox{0.90}{$
    \begin{bNiceMatrix}[c,margin]
    \CodeBefore
    \rectanglecolor{black!10}{1-1}{3-3}
    \rectanglecolor{black!10}{1-5}{3-6}
    \rectanglecolor{black!10}{1-8}{3-8}
    \rectanglecolor{black!10}{5-8}{6-8}
    \rectanglecolor{black!10}{5-1}{6-3}
    \rectanglecolor{black!10}{8-1}{8-3}
    \rectanglecolor{black!10}{8-5}{8-6}
    \rectanglecolor{black!10}{5-5}{6-6}
    \rectanglecolor{black!10}{8-8}{8-8}
    \Body
    1 & 0 & 0 & 0 & 0 & 0 & 0 & 1 & 0 & 0\\
    0 & 0 & 0 & 0 & 0 & 0 & 0 & 0 & 0 & 0\\
    0 & 0 & 0 & 0 & 0 & 0 & 0 & 0 & 0 & 0\\
    0 & 0 & 0 & 0 & 0 & 0 & 0 & 0 & 0 & 0\\
    0 & 0 & 0 & 0 & 1 & 0 & 0 & 0 & 0 & 1\\
    0 & 0 & 0 & 0 & 0 & 0 & 0 & 0 & 0 & 0\\
    0 & 0 & 0 & 0 & 0 & 0 & 0 & 0 & 0 & 0\\
    1 & 0 & 0 & 0 & 0 & 0 & 0 & 1 & 0 & 0\\
    0 & 0 & 0 & 0 & 0 & 0 & 0 & 0 & 0 & 0\\
    0 & 0 & 0 & 0 & 1 & 0 & 0 & 0 & 0 & 1
    \end{bNiceMatrix}$},
    &
    \mommat^\omega_{\Delta_{2}}(g_2\vec{y})
    &= \scalebox{0.90}{$
    \begin{bmatrix}
    2 & 0 & 0 & 0 & 0 & 2\\
    0 & 0 & 0 & 0 & 0 & 0\\
    0 & 0 & 0 & 0 & 0 & 0\\
    0 & 0 & 0 & 2 & 0 & 0\\
    0 & 0 & 0 & 0 & 0 & 0\\
    2 & 0 & 0 & 0 & 0 & 2
    \end{bmatrix}$},
    \\[1ex]
    \mommat_{\Delta_3}^\omega(\vec{y})
    &=\scalebox{0.90}{$
    \begin{bNiceMatrix}[c,margin]
    \CodeBefore
    \rectanglecolor{black!10}{1-1}{3-3}
    \rectanglecolor{black!10}{1-5}{3-6}
    \rectanglecolor{black!10}{1-8}{3-8}
    \rectanglecolor{black!10}{5-8}{6-8}
    \rectanglecolor{black!10}{5-1}{6-3}
    \rectanglecolor{black!10}{8-1}{8-3}
    \rectanglecolor{black!10}{8-5}{8-6}
    \rectanglecolor{black!10}{5-5}{6-6}
    \rectanglecolor{black!10}{8-8}{8-8}
    \Body
    1 & 0 & 1 & 0 & 0 & 0 & 0 & 0 & 0 & 0\\
    0 & 1 & 0 & 1 & 0 & 0 & 0 & 0 & 0 & 0\\
    1 & 0 & 1 & 0 & 0 & 0 & 0 & 0 & 0 & 0\\
    0 & 1 & 0 & 1 & 0 & 0 & 0 & 0 & 0 & 0\\
    0 & 0 & 0 & 0 & 0 & 0 & 0 & 0 & 0 & 0\\
    0 & 0 & 0 & 0 & 0 & 0 & 0 & 0 & 0 & 0\\
    0 & 0 & 0 & 0 & 0 & 0 & 0 & 0 & 0 & 0\\
    0 & 0 & 0 & 0 & 0 & 0 & 0 & 0 & 0 & 0\\
    0 & 0 & 0 & 0 & 0 & 0 & 0 & 0 & 0 & 0\\
    0 & 0 & 0 & 0 & 0 & 0 & 0 & 0 & 0 & 0
    \end{bNiceMatrix}$},
    &
    \mommat^\omega_{\Delta_{3}}(g_3\vec{y})
    &= \scalebox{0.90}{$
    \begin{bmatrix}
    2 & 0 & 2 & 0 & 0 & 0\\
    0 & 2 & 0 & 0 & 0 & 0\\
    2 & 0 & 2 & 0 & 0 & 0\\
    0 & 0 & 0 & 0 & 0 & 0\\
    0 & 0 & 0 & 0 & 0 & 0\\
    0 & 0 & 0 & 0 & 0 & 0
    \end{bmatrix}$},
\end{align}
\end{subequations}
which are positive semidefinite. The shaded parts of the moment matrices are the matrices $\smash{\mommat_{\Delta_i}^{\omega-d_i}(\vec{y})}$ appearing in condition \cref{e:flat-moments}; in this example, $d_1=d_2=d_3=1$. The `overlap' moment matrices appearing in \cref{e:flat-overlap} are
\begin{equation}
    \mommat_{\Delta_1\cap\Delta_2}^\omega(\vec{y})
    =
    \begin{bNiceMatrix}[c,margin]
    \CodeBefore
    \rectanglecolor{black!10}{1-1}{3-3}
    \Body
    1 & 0 & 1 & 0\\
    0 & 1 & 0 & 1\\
    1 & 0 & 1 & 0\\
    0 & 1 & 0 & 1
    \end{bNiceMatrix}
    \quad\text{and}\quad
    \mommat_{\Delta_2\cap\Delta_3}^\omega(\vec{y})
    =
    \begin{bNiceMatrix}[c,margin]
    \CodeBefore
    \rectanglecolor{black!10}{1-1}{3-3}
    \Body
    1 & 0 & 0 & 0\\
    0 & 0 & 0 & 0\\
    0 & 0 & 0 & 0\\
    0 & 0 & 0 & 0
    \end{bNiceMatrix}.
\end{equation}
The shaded parts correspond to the submatrices $\smash{\mommat_{\Delta_1\cap\Delta_2}^{\omega-1}(\vec{y})}$ and $\smash{\mommat_{\Delta_2\cap\Delta_3}^{\omega-1}(\vec{y})}$.  

\subsubsection{Detecting finite convergence}

Using the moment matrices reported above, one can check that the rank conditions in \cref{e:flat-moments,e:flat-overlap} hold with
\begin{subequations}\label{e:ex2-ranks}
    \begin{align}
        \rank\mommat_{\Delta_{1}}^{\omega}(\vec{y}) &=
        \rank\mommat_{\Delta_{1}}^{\omega-d_{1}}(\vec{y}) = 4,
        \\
        \rank\mommat_{\Delta_{2}}^{\omega}(\vec{y}) &=
        \rank\mommat_{\Delta_{2}}^{\omega-d_{2}}(\vec{y}) = 2,
        \\
        \rank\mommat_{\Delta_{3}}^{\omega}(\vec{y}) &=
        \rank\mommat_{\Delta_{3}}^{\omega-d_{3}}(\vec{y}) = 2,
        \\
        \rank\mommat_{\Delta_{1}\cap \Delta_{2}}^{\omega}(\vec{y}) &=
        \rank\mommat_{\Delta_{1}\cap \Delta_{2}}^{\omega-1}(\vec{y}) = 2,
        \\
        \rank\mommat_{\Delta_{2}\cap \Delta_{3}}^{\omega}(\vec{y}) &=
        \rank\mommat_{\Delta_{2}\cap \Delta_{3}}^{\omega-1}(\vec{y}) = 1.
    \end{align}
\end{subequations}
We then conclude from \Cref{thm:FF-POP} that the sparse moment relaxation for $\omega=3$ is exact. On the other hand, the finite convergence of the relaxation for the given vector $\vec{y}$ cannot be detected using either \cite[Theorem~3.7]{Lasserre2006} (the overlap matrix $\mommat_{\Delta_{1}\cap \Delta_{2}}^{\omega}(\vec{y})$ does not have rank one) or \cite[Theorem~3.5]{Nie2024} (the moment matrices have different rank). This remains true for $\omega$ up to $10$, which we tested numerically. We expect the same for all larger $\omega$ because, heuristically, the algorithm implemented in MOSEK appears to return moment matrices of maximum rank, in which case \cref{e:ex2-ranks} remains the same for all $\omega$. For this example, therefore, it seems essential to use the more general conditions established in this work.

\begin{figure}
    \centering
    \includegraphics[scale=0.85]{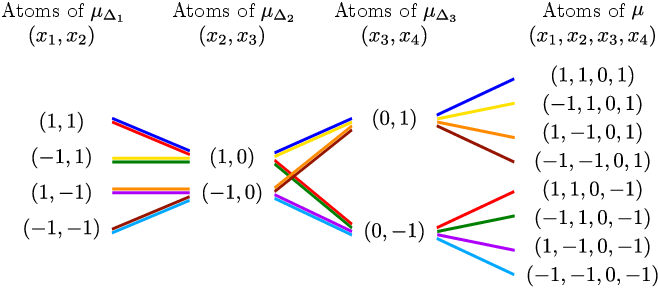}
    \caption{Atom concatenation strategy for the example of \cref{ss:ex-pop}. An atom $(a,b)$ for $\mu_{\Delta_i}$ is concatenated with an atom $(c,d)$ for $\mu_{\Delta_{i+1}}$ if $b=c$. Each colored path leads to a different minimizer for \cref{e:ex-pop}, which is an atom of the representing measure $\mu$ in~\cref{e:ex-atomic-mu}.}
    \label{fig:ex-pop}
\end{figure}

\subsubsection{Recovery of POP minimizers}
Having detected finite convergence, we recover an atomic representing measure for $\vec{y}$, whose atoms are optimizers for \cref{e:ex-pop}. Specifically, we first apply the atom extraction procedure of \cite{Henrion2005extraction} to each moment matrix $\smash{\mommat_{\Delta_i}^\omega(\vec{y})}$ to recover the `clique' atomic measures
\begin{subequations}
    \begin{gather}
        \mu_{\Delta_1} = 
        \tfrac14 \delta_{(1,1)} + 
        \tfrac14 \delta_{(1,-1)} + 
        \tfrac14 \delta_{(-1,-1)} + 
        \tfrac14 \delta_{(-1,1)},
        \\
        \mu_{\Delta_2} = 
        \tfrac12 \delta_{(1,0)} +
        \tfrac12 \delta_{(-1,0)}
        \\
        \mu_{\Delta_3} = 
        \tfrac12 \delta_{(0,1)} +
        \tfrac12 \delta_{(0,-1)}.
    \end{gather}
\end{subequations}
We then assemble these measures as described in \cref{ss:explicit-recovery}, concatenating their atoms in all compatible ways as illustrated in \Cref{fig:ex-pop}. The process also produces the corresponding weights in the atomic representing measure
\begin{equation}\label{e:ex-atomic-mu}
    \mu = \tfrac18 \sum_{i=1}^8 \delta_{\vec{x}^*_i}.
\end{equation}
Note that this measure is supported on all minimizers of the POP \cref{e:ex-pop}.

We stress that recovering the weights of the representing  measure is not necessary if one only cares about finding a POP minimizer. For this, once the conditions of \Cref{thm:FF-POP} hold, it suffices to find any feasible concatenation of the `clique atoms' $\vec{z}^i_j$ in \cref{e:local-measure}. Indeed, by the maximality of the support proved in \Cref{prop:supp-maximality}, this vector will be an atom of the representing measure and thus a POP minimizer.

\subsubsection{Further comments}

A remarkable feature of the present example is that, by exploiting correlative sparsity and using \Cref{thm:FF-POP}, we not only detect finite convergence and recover minimizers for the POP \cref{e:ex-pop}, but also do so with a relaxation order smaller than that typically required by a dense (i.e., single-clique) moment relaxation. Indeed, interior-point algorithms for the dense relaxation often produce an optimal $\vec{y}$ that maximizes the rank of the moment matrix among all optimal solutions. For such `maximal-rank solutions', by \cite[Theorems 3.18(iii) and 3.19]{Baldi2025}, flat rank conditions that certify finite convergence can hold at relaxation order $\omega$ only if, for each of the minimizers $\vec{x}^*_1,\ldots,\vec{x}^*_8$, there exists an `interpolator polynomial' $p_i$ of degree at most $\omega-1$ that satisfies $p_i(\vec{x}^*_i)=1$ and vanishes at all other minimizers. For our example, this requires $\omega \geq 4$.
 
On the other hand, sparse relaxations do not always enjoy this advantage. For instance,
\begin{equation}\label{e:conj-counterex}
  f^*=\min_{\vec{x}\in\mathbb{R}^3}\;
  (x_1-x_2)^4+(x_3-x_2^2)^2+x_2^2(x_3-1)^2+x_3^2(x_3-1)^2
\end{equation}
is a POP which attains its minimum $f^*=0$ exactly at $\vec{x}^*_1 =(-1,-1,1)$, $\vec{x}^*_2 =(0,0,0)$ and $\vec{x}^*_3 =(1,1,1)$. Its objective is a sum of squares of quadratic polynomials supported on the cliques $\Delta_1=\{1,2\}$ and $\Delta_2=\{2,3\}$, so $f^*_\omega=0$ for all $\omega\geq2$ both for the sparse relaxation \cref{e:mom} with these cliques and for the dense relaxation, which is \cref{e:mom} with the single clique $\Delta=\{1,2,3\}$. At $\omega=2$, the moments $\vec{y}$ of the measure $\frac13\delta_{\vec{x}_1^*}+\frac13\delta_{\vec{x}_2^*}+\frac13\delta_{\vec{x}_3^*}$ form a maximal-rank optimal solution of both relaxations. Since $\rank\mommat^2_\Delta(\vec{y})=\rank\mommat^1_\Delta(\vec{y})=3$, the standard flat rank conditions certify finite convergence of the dense relaxation at $\omega=2$. For the sparse relaxation, however, $\rank\mommat^2_{\Delta_1\cap\Delta_2}(\vec{y})=3>2=\rank\mommat^1_{\Delta_1\cap\Delta_2}(\vec{y})$, so \cref{e:flat-overlap} fails and \cref{thm:FF-POP} does not apply at this order. 

In this example, flat rank conditions for dense relaxations can hold at $\omega=2$ because the minimizers $\vec{x}^*_1$, $\vec{x}^*_2$ and $\vec{x}^*_3$ are affinely independent, and thus admit interpolator polynomials of degree one. In contrast, for the sparse relaxation, the minimizers' projections onto $\Delta_1\cap\Delta_2=\{2\}$ are three distinct points of the real line and require quadratic interpolators. For maximal-rank solutions, therefore, \cref{e:flat-overlap} can hold only if $\omega\geq3$.  The same obstruction arises whenever a POP has affinely independent minimizers with more than $|\Delta_i\cap\Delta_j|+1$ distinct projections onto a clique intersection $\Delta_i\cap\Delta_j$: such projections are necessarily affinely dependent, so they admit no interpolators of degree one even though the minimizers themselves do. It would be interesting to identify regimes in which flat rank conditions can detect the finite convergence of sparse relaxations at smaller relaxation orders than for dense relaxations, and regimes in which they cannot.

%%%%%%%%%%%%%%%%%%%%%%%%%%%%%%%%%%%%%%%%%%%%%%%%%%%%%%%%%%%%
\subsection{Necessity of the running intersection property}
\label{ss:counterexample}

Our last example is a correlatively sparse truncated moment problem on $K=\mathbb{R}^3$ showing that one cannot drop \cref{e:rip} from \Cref{thm:FF}. (A different example in the context of minimizer extraction for POPs is \cite[Example~3.4]{Nie2024}, which appeared after a preliminary version of the present work.)

\enlargethispage{0.15\baselineskip} %% NEEDED FOR GOOD PAGINATION ON ARXIV VERSION
Fix $n=3$ and consider the cliques $\Delta_{1}=\{1,2\}$, $\Delta_{2}=\{2,3\}$ and $\Delta_{3}=\{1,3\}$, which do \emph{not} satisfy \cref{e:rip}. Fix $\omega=2$ and consider a correlatively sparse vector $\vec{y} = (y_\alpha)_{\alpha \in \SparseExponents{4}} $ whose $31$ entries are $y_{000} = 1$ and
\begin{equation}
    \begin{aligned}
    y_{100} &= 0, & 
    y_{010} &= 0, & 
    y_{001} &= 0, & 
    y_{200} &= 1, & 
    y_{110} &= 1, &
    y_{020} &= 1, \\
    y_{101} &= -1, & 
    y_{011} &= 1, & 
    y_{002} &= 1, & 
    y_{300} &= 0, &
    y_{210} &= 0, & 
    y_{120} &= 0, \\
    y_{030} &= 0, & 
    y_{201} &= 0, & 
    y_{021} &= 0, &
    y_{102} &= 0, & 
    y_{012} &= 0, & 
    y_{003} &= 0, \\
    y_{400} &= 1, & 
    y_{310} &= 1, &
    y_{220} &= 1, & 
    y_{130} &= 1, & 
    y_{040} &= 1, & 
    y_{301} &= -1, \\ 
    y_{031} &= 1, &
    y_{202} &= 1, & 
    y_{022} &= 1, & 
    y_{103} &= -1, & 
    y_{013} &= 1, & 
    y_{004} &= 1.
    \end{aligned}
\end{equation}
A straightforward computation confirms that
\begin{subequations}
    \label{e:ex-atoms}
    \begin{align}
    \vec{y}_{\Delta_{1}} &=
    \tfrac{1}{2}\dmon{(1,1)}_{4} +
    \tfrac{1}{2}\dmon{(-1,-1)}_{4},\\
    \vec{y}_{\Delta_{2}} &=
    \tfrac{1}{2}\dmon{(1,1)}_{4} +
    \tfrac{1}{2}\dmon{(-1,-1)}_{4},\\
    \vec{y}_{\Delta_{3}} &=
    \tfrac{1}{2}\dmon{(1,-1)}_{4} +
    \tfrac{1}{2}\dmon{(-1,1)}_{4},
    \end{align}
\end{subequations}
which immediately implies that the moment matrices $\mommat^\omega_{\Delta_i}(\vec{y})$ are positive semidefinite (we need not consider localizing matrices in this example because we chose $K=\mathbb{R}^3$).
With $\omega=2$ and $d_1=d_2=d_3=1$, one also finds that
\begin{subequations}
    \begin{gather}
    \rank\mommat_{\Delta_{i}}^{\omega}(\vec{y})=\rank\mommat_{\Delta_{i}}^{\omega-d_{i}}(\vec{y}) =2
    \quad\forall i \in [3],\\
    \label{e:ex-rank-conditions}
    \rank\mommat_{\Delta_{i} \cap \Delta_{j}}^{\omega}(\vec{y})=\rank\mommat_{\Delta_{i} \cap \Delta_{j}}^{\omega-1}(\vec{y}) =2\quad 
    \forall i,j \in [3].
    \end{gather}
\end{subequations}
(Of course, one could verify these rank conditions directly from the definition of the moment matrices and only then recover \cref{e:ex-atoms} with the atom extraction procedure of~\cite{Henrion2005extraction}.)
These rank conditions are the analogue of those in \Cref{thm:FF} when \cref{e:rip} does not hold; note how the rank conditions for the `overlap' moment matrices are satisfied by any pair of intersecting cliques.

Despite this, and despite the existence of unique representing measures for the clique subvectors $\vec{y}_{\Delta_1}$, $\vec{y}_{\Delta_2}$ and $\vec{y}_{\Delta_3}$, the vector $\vec{y}$ has no atomic representing measure. Indeed, any atom $\vec{x}=(x_1,x_2,x_3)$ in such a representing measure would have to satisfy
\begin{subequations}
    \begin{align}
    \proj{\Delta_{1}}\vec{x} = (x_1, x_2) & \in\{(1,1),(-1,-1)\},\\
    \proj{\Delta_{2}}\vec{x} = (x_2, x_3) & \in\{(1,1),(-1,-1)\},\\
    \proj{\Delta_{3}}\vec{x} = (x_1, x_3) & \in\{(1,-1),(-1,1)\},
    \end{align}
\end{subequations}
which is impossible.

\subsection*{Acknowledgments}
We thank Jiawang Nie for motivating us to study the correlatively sparse truncated moment problem after reading a first draft of this work. We also thank Lorenzo Baldi for pointing us to the notion of core varieties and for highlighting the relevance of interpolator polynomials.

GF was partially supported by the DFG-ANR project MONET (DFG project number 568735456).
FF was partially supported by the National Center for Artificial Intelligence CENIA FB210017, Basal ANID based in Chile, by the Fondecyt Grant N.~11261732 from ANID in Chile, and by the Office of Naval Research (ONR) award N629092312098, based in the United States.
Both authors gratefully acknowledge support from a SQuaRE at the American Institute of Mathematics. 

% Bibliography
\bibliographystyle{abbrvnat}
\bibliography{references}

%%%%%%%%%%%%%%%%%%%%%%%%%%%%%%%%%%%%%%%%%%%%%%%%%%%%%%%%%%%%%%%%%%%%%%%%%%%%%%
\end{document}